%% file: hamming_torus.tex
\documentclass{article}
    \usepackage{amsfonts} 
\usepackage{amsmath}
\usepackage{amssymb}
\usepackage{amsthm}
\usepackage{subfigure}
\usepackage{geometry}
\usepackage{graphicx}
\usepackage{bbm}

\newtheorem{thm}{Theorem}[section]
\newtheorem{prop}[thm]{Proposition}
\newtheorem{lem}[thm]{Lemma}
\newtheorem*{theorem}{Theorem}
\newtheorem*{lemma}{Lemma}

     \newcommand{\norm}[1]{\left\|#1\right\|}
     \newcommand{\abs}[1]{\left|#1\right|}
     \newcommand{\ip}[2]{\left \langle #1, #2 \right \rangle}

     \newcommand{\prob}[2]{\mathbb{P}_{#2} \left(#1\right)}

     \newcommand{\E}[0]{\mathbb{E}}
     \newcommand{\Var}[0]{\text{Var}}

     \newcommand{\A}[0]{\mathcal{A}}
          
     \newcommand{\vect}[1]{\boldsymbol{#1}}
     
     \newcommand{\deq}[0]{\ {\buildrel d \over =}\  }
     \def\gae{\lower 3pt \hbox{$\ \buildrel {\displaystyle >}\over \sim \ $}} 
     \def\lae{\lower 3pt \hbox{$\ \buildrel {\displaystyle <}\over \sim \ $}} 
          
\begin{document}
  
    \title{\bf Emergence of a Giant Component in Random Site Subgraphs of a $d$-Dimensional Hamming Torus}
    
\author{David Sivakoff\footnote{Dissertation research at UC Davis supported in part by NSF VIGRE Grant No. DMS-0636297, and NSF Grant Nos. DMS-0805970 and  DMS-0505734.} \\ Department of Mathematics, University of California, Davis.}

    \maketitle
 
 \begin{abstract}
The $d$-dimensional Hamming torus is the graph whose vertices are all of the integer points inside an $a_1 n \times a_2 n \times \cdots \times a_d n$ box in $\mathbb{R}^d$ (for constants $a_1, \ldots, a_d >0$), and whose edges connect all vertices within Hamming distance one.  We study the size of the largest connected component of the subgraph generated by independently removing each vertex of the Hamming torus with probability $1-p$.  We show that if $p= \frac{\lambda}{n}$, then there exists $\lambda_c >0$, which is the positive root of a degree $d$ polynomial whose coefficients depend on $a_1, \ldots, a_d$, such that for $\lambda < \lambda_c$ the largest component has $O(\log n)$ vertices (a.a.s. as $n \to \infty$), and for $\lambda > \lambda_c$ the largest component has  $(1-q) \lambda \left(\prod_i a_i \right) n^{d-1} + o (n^{d-1})$ vertices and the second largest component has $O(\log n)$ vertices (a.a.s.).  An implicit formula for $q < 1$ is also given.  Surprisingly, the value of $\lambda_c$ that we find is distinct from the critical value for the emergence of a giant component in the random edge subgraph of the Hamming torus.  Additionally, we show that if $p = \frac{c \log n}{n}$, then when $c < \frac{d-1}{\sum a_i}$ the site subgraph of the Hamming torus is not connected, and when $c > \frac{d-1}{\sum a_i}$ the subgraph is connected (a.a.s.).  We also show that the subgraph is connected precisely when it contains no isolated vertices.
\end{abstract}

 %
 %
    \section{Introduction}
Erd\"os and R\'enyi studied random subgraphs of the complete graph in~\cite{E-R:1960} (for an account of their results, see~\cite{durrett:2007}).  In the Erd\"os-R\'enyi model, a random subgraph of the complete graph on $n$ nodes is obtained by independently deciding whether to remove each edge with probability $(1-p)$ or keep it with probability $p$. Erd\"os and R\'enyi studied the size of the largest connected component in the random subgraph under the scaling $p = \lambda/n$, where $\lambda$ is a constant parameter.  They found that the size of the largest connected component is asymptotically almost surely $O(\log n)$ for $\lambda < 1$, and $cn + o(n)$ for $c = c(\lambda)>0$ when $\lambda>1$.  A sequence of events $E_n$ is said to occur asymptotically almost surely (a.a.s.) if $\prob{E_n}{}\rightarrow 1$ as $n\rightarrow \infty$.  Thus, at $\lambda = 1$, the  Erd\"os-R\'enyi model undergoes a transition from having small components to having a giant connected component.
            
Random subgraphs generated by randomly removing edges independently with probability $(1-p)$ have since been extensively studied on other graphs $G = (V,E)$, and we will refer to random subgraphs generated in this manner as random edge subgraphs.  Alternatively, random subgraphs of $G$ can be obtained by independently removing vertices (and any edges incident to those vertices) with probability $(1-p)$; we will refer to these as random site subgraphs.
            
The $n$-dimensional hypercube is an example of a graph for which both random edge and site subgraphs have been studied.  Ajtai et~al.~\cite{AKS:1981} studied random edge subgraphs of the $n$-dimensional hypercube:  $V = \{0,1\}^n$, $E = \{(x,y)\in V \times V : d(x,y) = 1 \}$, where $d(x,y)$ is the Hamming distance between $x$ and $y$ (the number of coordinates in which they differ).  For the scaling $p = \lambda/n$, they proved that all components of the random edge subgraph have $O(n)$ vertices  when $\lambda < 1$ (a.a.s.), and for $\lambda > 1$ there is a component with $c2^n + o(2^n)$ vertices (a.a.s.) for $c = c(\lambda)>0$.  Bollob\'as et. al.~\cite{BKL:1991} found similar behavior for the random site subgraph of the $n$-dimensional hypercube.  Under the same scaling, they proved that for $\lambda<1$ all components have at most $O(n)$ vertices (a.a.s.), and for $\lambda>1$ there is a component with $cn^{-1}2^n + o(n^{-1} 2^n)$ vertices for $c = c(\lambda)>0$ (a.a.s.).  Thus, both random subgraph models have the same threshold under the scaling $p = \lambda /n$ for the $n$-dimensional hypercube.  In fact, the value of $c(\lambda)$ is the same for both models, so even the proportion of vertices in the giant component above the threshold is the same.

The intuition behind each of these results is that random edge subgraphs of the hypercube and the complete graph and random site subgraphs of the hypercube look locally tree-like.  So, the connected component associated with a fixed vertex can be compared to a branching process with a Binomial$(D,p)$ offspring distribution, where $D$ is the degree of each vertex.  A giant connected component exists precisely when the branching process has positive survival probability, which is when the expected number of offspring per individual ($D p = \lambda$) exceeds $1$.  Borgs et~al.~\cite{lace1,lace2} showed that, for a wide variety of finite transitive graphs, the threshold for the existence of a giant component in the random edge subgraph should be at $D p = 1$, though they focus on behavior in the critical window, and do not give a lower bound for the size of the supercritical giant component.  The case of random site subgraphs is less clear --- just consider the complete graph.  All random site subgraphs of the complete graph are connected, thus the threshold for the existence of a giant component is $0$.  This is a trivial case, but suggests that the random edge and site subgraphs will behave differently when the edge density of $G$ is closer to that of the complete graph than to the hypercube.

Random site subgraphs are particularly relevant to problems that arise from the theory of genetic fitness landscapes.  Fitness landscapes are a tool used in theoretical evolutionary biology to model speciation under various conditions~\cite{Gav:2004}.  A fitness landscape consists of a genotype space and a function that maps each combination of genes to a fitness level in the interval $[0,1]$, which can be interpreted as the probability of an individual with that genotype surviving to reproduce.  For example, the genotype space for a haploid genome  with $n$ diallelic loci (two gene alternatives at each position along the genome) can be represented by the $n$-dimensional hypercube, where each vertex represents a genotype and edges represent single-gene mutations.  Each genotype is then assigned a fitness level of $1$ with probability $p$ or $0$ with probability $(1-p)$, independent of all other genotypes --- this is equivalent to taking a random site subgraph.  For this model, provided $n$ is sufficiently large, only a small proportion of genotypes need to be viable ($p>1/n$) for a significant number of viable genotypes to be accessible via single-gene mutations.

The random graph model that we are interested in comes from a genotype space for a haploid genome with $d$ loci, and $a_i n$ possible alleles at the $i^{th}$ locus.  Here $d$ is fixed, and the number of alleles at each locus is assumed to be large.  The resulting graph is a Hamming torus in $d$ dimensions:
\begin{align*}
	V =& \left\{ (x_1, \ldots, x_d)\in \mathbb{Z}^d : 1\leq x_i \leq a_i n, \ i = 1, \ldots, d \right\} \\
	E =&  \left\{ (x,y) \in V \times V : d(x,y) = 1 \right\},
\end{align*}
where $d(x,y)$ is the Hamming distance between $x$ and $y$.  For $d = 1$ this is the complete graph on $a_1 n$ vertices, and for $d=2$ this is the rook graph on an $a_1 n$ by $a_2 n$ chessboard.  As before, fitness levels of $1$ are assigned independently with probability $p$, and fitness of $0$ is otherwise inferred.  We say that sites with a fitness of $1$ are {\em occupied}, while those with a fitness of $0$ are {\em removed} or unoccupied.

This model is amenable to comparison with a multitype branching process.  Consider a multitype branching process with $d$ types in which an individual of type $i$ has a $\text{Binomial}(a_j n, \frac{\lambda}{n})$ number of offspring of type $j$ for $j \neq i$ and zero offspring of type $i$.  This branching process has nonzero survival probability iff the largest eigenvalue of the matrix of expected progeny, $\vect{M}_{\lambda}$, is strictly larger than one~\cite{branching}.

\begin{equation*}
\vect{M}_{\lambda} = \left( \begin{array}{ccccc}
0 & \lambda a_2 & \lambda a_3 &\cdots& \lambda a_d  \vspace{.1cm} \\
\lambda a_1 & 0 & \lambda a_3 &\cdots& \lambda a_d\\
\lambda a_1 & \lambda a_2 & 0 & & \vdots \\
\vdots & \vdots & & \ddots & \lambda a_d\\
\lambda a_1 & \lambda a_2 & \cdots & \lambda a_{d-1} & 0
\end{array} \right)
\end{equation*}
where the entry in position $(i,j)$ is the expected number of type $j$ offspring to which a type $i$ individual will give birth.  This branching process is a good local approximation for clusters in the random site subgraph of the Hamming torus when $p = \frac{\lambda}{n}$, which motivates the following theorems.

\begin{thm}
\label{thma}
Fix $d\geq 2$ and $a_1, \ldots, a_d>0$, and let $\lambda_c$ be the unique positive solution to $\det \left(\vect{M}_{\lambda_c} - \vect{I}\right) = 0$.  If $\lambda < \lambda_c$ and $p = \frac{\lambda}{n}$ then the largest connected component of the random site subgraph of the Hamming torus is a.a.s. $O( \log n)$.
\end{thm}

\begin{thm}
\label{thmb}
Fix $d \geq 2$ and $a_1, \ldots, a_d > 0$, and let $\lambda_c$ be the unique positive solution to $\det \left(\vect{M}_{\lambda_c} - \vect{I} \right) = 0$.  If $\lambda > \lambda_c$ and $p = \frac{\lambda}{n}$ then the size of the largest connected component of the random site subgraph of the Hamming torus is a.a.s. $(1-q) \lambda \left(\prod_i a_i \right) n^{d-1} + o (n^{d-1})$.  Furthermore, the second largest component has a.a.s. $O(\log n)$ vertices.
\end{thm}

\noindent For Theorem~\ref{thmb}, $(1-q)>0$ is defined precisely in Section~\ref{section_supercritical4}, and in both theorems $\vect{I}$ is the $d\times d$ identity matrix.  Note that the total number of vertices in the random site subgraph of the Hamming torus is asymptotically the expected number, $p \abs{V} = \lambda \left(\prod_i a_i \right) n^{d-1}$, so $(1-q)$ is the proportion of vertices remaining that are in the giant component.  The critical value, $\lambda_c$, can be defined equivalently as the root of a degree $d$ polynomial due to the following lemma.

\begin{lem}
\label{polynomialLemma}
\begin{equation*}
\det(\vect{M}_{\lambda} - \vect{I}) = (-1)^d \left[ 1 - \sum_{\ell=2}^{d} (\ell-1) \lambda^{\ell} \mathop{\sum_{S \subset \{1, \ldots, d\} }}_{\abs{S} = \ell} \ \prod_{i \in S} a_{i} \right].
\end{equation*}
\end{lem}

Random {\em edge} subgraphs of the Hamming torus were considered by Borgs et~al.~\cite{lace1,lace2}, who indicated that the threshold for the emergence of a giant component in the random edge subgraph should be $1 / (a_1 + \cdots + a_d)$, though they do not provide a lower bound for the size of the largest component above the threshold.  This lower bound was later proven by van der Hofstad and Luczak~\cite{HL:2009} in the case $d=2$ and $a_1 = a_2=1$, demonstrating that the threshold occurs at $1/2$. When $d = 2$, the threshold for the random site subgraph is $\lambda_c =1/\sqrt{a_1 a_2}$, so the two processes clearly differ.  

Multitype branching processes have also been employed in the analysis of {\em inhomogeneous} random edge subgraphs of the complete graph by Bollob\'as et al. \cite{BJR:2007}.  In the inhomogeneous model, the probability of retaining the edge between vertices $i$ and $j$ is $p_{ij}$, and the $p_{ij}$'s may not be equal, however the inclusion or exclusion of each edge still occurs independently of all other edges.  Since the number of neighbors that a vertex has depends on the probabilities of the edges incident to that vertex, the inhomogeneous model looks locally like a multitype branching process --- each vertex clearly has a different offspring distribution, and is thus of a different ``type''.  Bollob\'as et al. \cite{BJR:2007} proved that, under certain conditions on the values of the $p_{ij}$'s (which exclude the case of the Hamming torus), a giant component exists precisely when the corresponding multitype branching process survives.  In the case of random site subgraphs of the Hamming torus, it is somewhat surprising that a connection to multitype branching processes arises, since each vertex sees the same distribution of neighbors.  In this model, the similarity to a multitype branching process is due to the dependencies between edges, and not the inhomogeneity of the edge probabilities as in~\cite{BJR:2007}.

In Section~\ref{section_subcritical} we prove Theorem~\ref{thma} by coupling a process of revealing the vertices in a connected component of the random site subgraph with a multitype branching process.  In Section~\ref{section_supercritical} we prove Theorem~\ref{thmb} in four steps.  The first step, which is in Section~\ref{section_supercritical2}, is to show that the process of revealing the vertices in a connected component of the random site subgraph will either terminate before discovering $O(\log n)$ occupied vertices, or will reach size $m = \Theta(\log n)$ with high probability.  The second step, in Section~\ref{section_supercritical1}, is to show that if the process of revealing vertices reaches size $m$, then it can be coupled with a lower bounding branching process, which will reach size $n^{d-4/3}$ with high probability.  The third step, in Section~\ref{section_supercritical3}, is to show that any two component-discovering processes (started from two different vertices in $V$) that reach size $n^{d-4/3}$ will join together with high probability.  The final step, in Section~\ref{section_supercritical4}, is to show that the proportion of vertices in components of size $O(\log n)$ converges to $q$ in probability.  The reason we reverse the order in which the first two steps are presented is because $m$ needs to be defined before proving the first step, but this definition is motivated by the second step.  In Section~\ref{section_connectivity} we will prove the following three theorems regarding the connectivity of the random site subgraph.

\begin{thm}
\label{thm-discon}
Let $c < \frac{d-1}{\sum a_i}$.  If $p = p(n) \leq \frac{c \log n}{n}$ and $p = \omega(n^{-d})$ then the random site subgraph of the Hamming torus contains isolated vertices, and is thus not connected (a.a.s.).
\end{thm}

\begin{thm}
\label{thm-con}
Let $c > \frac{d-1}{\sum a_i}$.  If $p = p(n) \geq \frac{c \log n}{n}$ then the random site subgraph of the Hamming torus is connected (a.a.s.).
\end{thm}

\begin{thm}
\label{thm-isocon}
Fix $a_1 \geq a_2 \geq \cdots \geq a_d$, and let $c > \frac{d-1}{2\sum_{i=2}^d a_i + a_1 }$.  If $p = \frac{c \log n}{n}$ then every vertex in the random site subgraph of the Hamming torus is either isolated or belongs to the giant component~(a.a.s.).
\end{thm}

These three theorems together imply that, with probability approaching one, the random site subgraph of the Hamming torus is connected if and only if it contains no isolated vertices (except in the trivial case $p \asymp n^{-d}$, where the subgraph may consist of just a single occupied vertex with positive probability).

\section{Subcritical behavior}
\label{section_subcritical}

Theorem~\ref{thma} comes from a direct comparison with the binomial multitype branching process described above, and thus follows from Proposition~\ref{propa} below, which bounds the total size of a multitype branching process.  Consider a multitype branching process with $d$ types, and let $\vect{Z}_0, \vect{Z}_1, \vect{Z}_2,\ldots$ be the generations of this process, where $\vect{Z}_t$ is a vector whose $i^{\text{th}}$ component, $Z_t^i$, is the number of type $i$ individuals in the $t^{\text{th}}$ generation.  $\vect{Z}_0$ is assumed to be deterministic, and all birth events are independent of one another.  Let $\vect{M} = (m_{ij})$ denote the matrix of expectations:
\begin{equation*}
m_{ij} = \E\left(Z_1^j \ \middle | \ \vect{Z_0} = \vect{e_i}\right).
\end{equation*}
We say that a matrix, $\vect{A}$, or a vector, $\vect{x}$, is {\em positive} if $A_{ij} > 0$ for all $i$ and $j$ or $x_i>0$ for all $i$.  Lemma~\ref{pflemma} is a simplified restatement of the Perron-Frobenius Theorem as it appears in~\cite{branching}.
\begin{lem}
\label{pflemma}
If $\vect{M}^N$ is positive for some natural number $N$ then $\vect{M}$ has a positive simple eigenvalue, $\rho$, that is greater in absolute value than any other eigenvalue, and $\rho$ corresponds to a positive right eigenvector $\vect{\mu}:\vect{M \mu} = \rho \vect{\mu}$.
\end{lem}

We wish to state Proposition~\ref{propa} in some generality, but we will only be applying it to the multitype branching process with expectation matrix $\vect{M}_{\lambda}$.   When $d \geq 3$, it is easy to check that $\vect{M}_\lambda^2$ is positive, and when $d=2$ it can be verified directly that $\vect{M}_\lambda$ has a unique positive eigenvalue corresponding to a positive right eigenvector.

\begin{prop}
\label{propa}
If $\vect{M \mu} = \rho \vect{\mu}$ where $0 < \rho < 1$ and $\vect{\mu}$ is positive, and there exists $\theta_0 >0$ such that $\E \left[ e^{\theta_0 \ip{\vect{Z_1}}{\vect{\mu} }} \middle | \vect{Z_0} = \vect{e_i}\right] < \infty$ for $i = 1, \ldots, d$, then $\prob{\sum_{t=0}^\infty \norm{\vect{Z}_t}_1 > x}{} \leq C e^{-\alpha x}$ where $\alpha, C > 0$ can be chosen as follows:
\begin{align*}
e^{-\alpha} := \min_{\theta \in [0, \theta_0]} \max_i \ \E \left[ e^{\theta \ip{\vect{Z_1} - \vect{e_i}}{\vect {\mu}} } \middle | \vect{Z_0} = \vect{e_i}  \right],
\end{align*}
and, letting $\theta'$ be the value of $\theta$ for which this minimum is attained, $C := \ e^{\theta' \ip{\vect{Z_0}}{\mu}}$.
\end{prop}


\begin{proof}[Proof of Proposition \ref{propa}]  Suppose $X \geq 0$ is an arbitrary random variable such that $\E e^{\theta_1 X} < \infty$ for some $\theta_0 > 0$.  By expanding, for $\theta < \theta_0$:
\begin{equation*}
\frac{d}{d\theta} \left. \left(\E e^{\theta X} \right) \right\vert_{\theta = 0} \ = \ \frac{d}{d\theta} \left. \left(1 + \theta \ \E X + \frac{1}{2} \theta^2 \ \E X^2 + \cdots \right) \right\vert_{\theta = 0} \ = \ \E X.
\end{equation*}
Thus, for each $i = 1, \ldots, d$, 
\begin{equation}
\label{mgf-deriv}
\frac{d}{d\theta} \left. \left( \E \left[ e^{\theta \ip{\vect{Z_1}}{\vect{\mu} }} \middle | \vect{Z_0} = \vect{e_i}\right] \right) \right\vert_{\theta=0} = \E \left[ \ip{\vect{Z_1}}{\vect{\mu}} \middle | \vect{Z_0} = \vect{e_i}\right] = \rho \mu_i .
\end{equation}

\noindent  Now consider a random walk version of the multitype branching process, $(\vect{S}_t)$, constructed as follows. $\vect{S}_0 = \vect{Z_0}$ is a non-random initial vector whose $i^\text{th}$ component indicates the number of type $i$ individuals who are {\em active}.  At each step, an active individual is chosen uniformly at random, it gives birth to a random number of individuals depending on its type and according to the law for that type in the branching process (these new individuals are considered active), then it is made inactive (thus no longer included in $\vect{S}_t$).
\begin{equation}
\label{upperBoundingRW}
\vect{S}_{t+1} = \vect{S}_t + \sum_{i = 1}^d \mathbbm{1}_{\{j_{t+1} = i\}} \left( \vect{X}_{t+1}^i - \vect{e_i} \right)
\end{equation}
Where $j_t$ is the random variable that takes the value $i$ if an active individual of type $i$ is selected at time $t$, so $\prob{j_{t+1} = i \ | \ \vect{S}_t = \vect{v}}{} = v_i / \norm{\vect{v}}_1$, and for each $i= 1, \ldots, d$ the random vectors $\vect{X}_t^i$ are independent and equal in distribution to $\vect{Z}_1$ conditional on $\vect{Z}_0 = \vect{e_i}$.  Notice that $j_{t+1}$ is dependent on $\vect{S}_t$, but $\vect{X}_{t+1}^i$ is not.  The process continues until the stopping time $T := \inf \left\{ t \ \middle | \ \vect{S}_t = \vect{0} \right\}$, at which time the process dies out.  Notice that $T = \sum_{t=0}^\infty \norm{\vect{Z}_t}_1$.

Now consider an increment of this process:
\begin{align}
\nonumber \phi_{\vect{v}} (\theta) &:= \E \left[ \exp\left(\theta \ip{\sum_i \mathbbm{1}_{\{j_{t+1} = i\}} \left( \vect{X}_{t+1}^i - \vect{e_i}\right)}{\vect{\mu}} \right) \ \middle | \ \vect{S}_t = \vect{v} \right] \\
\nonumber &= \sum_i \frac{v_i}{\norm{\vect{v}}_1} \ \E e^{\theta \ip{(\vect{X}_{t+1}^i - \vect{e}_i)}{\vect{\mu}}} \\
 \label{mgf-bound} &\leq \max_i \ \E e^{\theta \ip{(\vect{X}_1^i - \vect{e}_i)}{\vect{\mu}}}\\
\nonumber &=: \psi(\theta).
\end{align}
Note that the last expression does not depend on the previous state vector, $\vect{v}$.  By equation (\ref{mgf-deriv}), for each $i = 1, \ldots, d$ we have:
\begin{equation*}
\frac{d}{d\theta} \left. \left( \E e^{\theta \ip{(\vect{X}_1^i - \vect{e}_i)}{\vect{\mu}}} \right) \right |_{\theta = 0} = \mu_i (\rho - 1).
\end{equation*}
Since $\rho < 1$ and $\mu_i > 0$ for all $i$, we have that there exist $\theta' , \alpha> 0$ such that $\psi(\theta') = e^{-\alpha} < 1$.  The following computation shows that $ e^{\theta \ip{\vect{S}_t}{\mu}} / \psi(\theta)^t$ is a supermartingale, and uses the estimate in (\ref{mgf-bound}). 
\begin{align*}
\E \left[ e^{\theta \ip{\vect{S}_{t+1}}{\vect{\mu}}} \ \middle | \ \vect{S}_t \right] &= \sum_{\vect{v} \geq \vect{0}} \mathbbm{1}_{\{\vect{S}_t = \vect{v}\} } e^{\theta \ip{\vect{v}}{\vect{\mu}}} \ \phi_{\vect{v}}(\theta) \\
&\leq \psi(\theta) \ e^{\theta \ip{\vect{S}_{t}}{\vect{\mu}}}
\end{align*}
So, by the Optional Stopping Theorem for positive supermartingales,
\begin{align*}
C := e^{\theta' \ip{\vect{S}_0}{\vect{\mu}}} \ \geq \  \E \left[ \frac{e^{\theta' \ip{\vect{S}_{T}}{\vect{\mu}}}}{\psi(\theta')^T} \right] \ = \ \E e^{\alpha T}
\end{align*}
then by Markov's inequality:
\begin{align*}
\prob{T > x}{} &= \prob{e^{\alpha T} > e^{\alpha x}}{} \\
& \leq C e^{-\alpha x}.
\qedhere
\end{align*}
\end{proof}

%
%

\begin{proof}[Proof of Theorem \ref{thma}] Consider a fixed vertex, $\vect{v}$, in the Hamming Torus, and let $C_v$ be the cluster containing that vertex ($C_{\vect{v}} = \emptyset$ if the vertex $\vect{v}$ is unoccupied).  We can reveal the vertices in $C_{\vect{v}}$ using the following cluster-discovering algorithm:
\begin{enumerate}
\item Initialize the sets $R_0 = \emptyset$, $A_0 = \{\vect{v}\}$, $U_0 = V \setminus A_0$.
\item Choose a vertex, $\vect{v}_t$, from $A_t$ in some measurable way (e.g., lexicographically or uniformly at random).
\item $A_{t+1} = A_t \setminus \{\vect{v}_t\} \cup \left\{ \vect{w} \in \mathcal{N}(\vect{v}_t) \cap U_t \ \middle | \ \vect{w} \text{ is occupied} \right\}$
\item $U_{t+1} = U_t \setminus \mathcal{N}(\vect{v}_t)$
\item $R_{t+1} = R_t \cup \{\vect{v}_t\}$
\item If $A_{t+1}$ is empty, then return $C_{\vect{v}} = R_{t+1}$, otherwise increment $t$ and go to step 2.
\end{enumerate}
In the above algorithm, $\mathcal{N}(\vect{v}) = \{ \vect{w} \in V \ | \ d(\vect{v},\vect{w}) = 1\}$ is the neighborhood of vertex $\vect{v}$.  To compare this process with a multitype branching process, we can say that at step 3, a neighbor of $\vect{v}_t$, say $\vect{w}$, is of type $i$ if $\vect{v}_t - \vect{w} = m \vect{e}_i$ for some integer $m$.  In future iterations of the algorithm, a type $i$ individual can only yield offspring of types $j \in \{1, \ldots, d \} \setminus \{ i\}$, since all of its neighbors in the $\vect{e_i}$ direction will have been removed from $U_{t+1}$.  Thus, $\abs{C_{\vect{v}}}$ is bounded above (in distribution) by the size of a multitype branching process with $Z_1^j$ distributed as a $\text{Binomial}(a_j n, \frac{\lambda}{n})$ random variable for $j \neq i$ and $Z_1^i \equiv 0$ when $\vect{Z}_0 = \vect{e_i}$.  The expectation matrix for this branching process is $\vect{M}_{\lambda}$, and the largest eigenvalue of $\vect{M}_{\lambda}$ is less than $1$ iff $\lambda < \lambda_c$.  In this regime, we have that the moment generating functions for $(\ip{\vect{Z_1} - \vect{e_i}}{\vect{\mu}} \ | \ \vect{Z_0} = \vect{e_i} )$ are uniformly (in $n$) bounded above for $\theta \geq 0$:
\begin{align*}
\E \left[ e^{\theta \ip{\vect{Z_1} - \vect{e_i}}{\mu}} \middle |  \vect{Z_0} = \vect{e_i}  \right] &= e^{-\theta \mu_i}\ \prod_{j \neq i} \ \sum_{k_j = 0}^{a_j n} { {a_j n}\choose{k_j} } \left(\frac{\lambda}{n}\right)^{k_j} \left(1-\frac{\lambda}{n}\right)^{a_j n - k_j} e^{\theta k_j \mu_j} \\
& = e^{-\theta \mu_i}\ \prod_{j \neq i} \left[ 1 + \frac{\lambda}{n} (e^{\theta \mu_j} - 1) \right]^{a_j n} \\
& \leq e^{ - \theta \mu_i } \prod_{j \neq i} \exp \left[\lambda a_j (e^{\theta \mu_j}- 1) \right] \\
& =: \psi_i (\theta).
\end{align*}
It is easy to verify that $\psi_i'(0) = \mu_i (\rho - 1) < 0$, $\psi_i (0) = 1$, $\psi_i(\theta) \to \infty$ as $\theta \to \infty$, and $\psi_i''(\theta) > 0$ for all $\theta >0$ which imply that we can choose $\alpha, C >0$ such that 
\begin{align*}
e^{-\alpha} := \min_{\theta \geq 0} \max_{i} \ \psi_i (\theta) &=: \max_{i} \ \psi_i (\theta'), \\
C :=& \exp \left[ \theta' \norm{ \vect{\mu}}_1 \right].
\end{align*}
Now applying Proposition \ref{propa} to this branching process yields the following inequality:
\begin{align*}
\prob{\abs{C_{\vect{v}}} > \frac{d+1}{\alpha} \log n}{} & \leq \prob{\sum_{t = 0}^{\infty} \norm{\vect{Z}_t}_1 > \frac{d+1}{\alpha} \log n}{}  \\
& \leq C n^{-d-1}.
\end{align*}
This implies the desired result:
\begin{equation*}
\prob{\max_{\vect{v} \in V} \abs{C_{\vect{v}}} > \frac{d+1}{\alpha} \log n}{}  \leq C(a_1 \cdots a_d) n^{-1}.
\qedhere
\end{equation*}
\end{proof}

%
%

\section{Supercritical behavior}
\label{section_supercritical}
Theorem~\ref{thmb} shows that $\lambda_c$ is the critical threshold for the emergence of a giant component.  If we let $q_i = \prob{\vect{Z}_t = \vect{0} \text{ for some t } \middle | \ \vect{Z}_0 = \vect{e_i} }{}$ be the extinction probabilities for the multitype branching process in which an individual of type $i$ gives birth to a $\text{Poisson}(\lambda a_j)$ number of type $j$ individuals for $j\neq i$ and zero individuals of type $i$, then the proportion of occupied vertices in the largest component when $\lambda > \lambda_c$ is $1 - \left( \prod_i q_i \right)^{1/(d-1)} \equiv 1-q > 0$.

\begin{theorem}[\bf\ref{thmb} restated]
Fix $d \geq 2$ and $a_1 \geq a_2 \geq \ldots \geq a_d > 0$, and let $\lambda_c$ be the unique positive solution to $\det(\vect{M}_{\lambda_c} - \vect{I}) = 0$.  If $\lambda > \lambda_c$ and $p = \frac{\lambda}{n}$ then the size of the largest connected component of the random site subgraph of the Hamming Torus is a.a.s. $(1-q) \lambda \left(\prod_i a_i \right) n^{d-1} + o (n^{d-1})$.  Furthermore, the second largest component has a.a.s. $O(\log n)$ vertices.
\end{theorem}


\noindent Before beginning the proof, it will be useful to show that the matrix equation defining $\lambda_c$ is equivalent to a polynomial equation that is easier to understand.
\begin{lemma}[\bf\ref{polynomialLemma} restated]
\begin{equation}
\label{poly}
\det(\vect{M}_{\lambda} - \vect{I}) = (-1)^d \left[ 1 - \sum_{\ell=2}^{d} (\ell-1) \lambda^{\ell} \mathop{\sum_{S \subset \{1, \ldots, d\} }}_{\abs{S} = \ell} \ \prod_{i \in S} a_{i} \right].
\end{equation}
\end{lemma}

\noindent From this reformulation, it is easy to see that a positive solution to $\det(\vect{M}_{\lambda} - \vect{I}) =0$ exists and is unique, since the polynomial in the above expression is monotone increasing for odd $d$ (monotone decreasing for even $d$) and $\lambda>0$, and the value at $\lambda = 0$ is $(-1)^d$.

\begin{proof}[{Proof of Lemma \ref{polynomialLemma}}] To demonstrate the validity of equation (\ref{poly}), we will first make a change of variables by letting $\gamma = 1/\lambda$, then multiplying both sides by $\gamma^d$, so the equation we must verify is:
\begin{equation}
\label{polygam}
\left| \begin{array}{ccccc}
-\gamma & a_2 & a_3 &\cdots& a_d  \vspace{.1cm} \\
a_1 & -\gamma & a_3 &\cdots& a_d\\
a_1 & a_2 & -\gamma & & \vdots \\
\vdots & \vdots & & \ddots & a_d\\
a_1 & a_2 & \cdots & a_{d-1} & -\gamma
\end{array} \right| = (-1)^d \left[ \gamma^{d} - \sum_{\ell=2}^{d} (\ell-1) \gamma^{d-\ell} \mathop{\sum_{S \subset \{1, \ldots, d\} }}_{\abs{S} = \ell} \ \prod_{i \in S} a_{i} \right].
\end{equation}
We will proceed by induction on $d$.  For $d=1$ both sides of equation (\ref{polygam}) are equal to $-\gamma$.  For $d > 1$, we will start by showing that the derivative with respect to $\gamma$ on both sides of equation (\ref{polygam}) is the same.  Taking the derivative of the left hand side yields:

\begin{align*}
\frac{d}{d\gamma} \left| \begin{array}{ccccc}
-\gamma & a_2 & a_3 &\cdots& a_d  \vspace{.1cm} \\
a_1 & -\gamma & a_3 &\cdots& a_d\\
a_1 & a_2 & -\gamma & & \vdots \\
\vdots & \vdots & & \ddots & a_d\\
a_1 & a_2 & \cdots & a_{d-1} & -\gamma
\end{array} \right| =& \left| \begin{array}{ccccc}              
-1 & 0 & 0 &\cdots& 0  \vspace{.1cm} \\
a_1 & -\gamma & a_3 &\cdots& a_d\\
a_1 & a_2 & -\gamma & & \vdots \\
\vdots & \vdots & & \ddots & a_d\\
a_1 & a_2 & \cdots & a_{d-1} & -\gamma
\end{array} \right|   + \left| \begin{array}{ccccc}            
-\gamma & a_2 & a_3 &\cdots& a_d  \vspace{.1cm} \\
0 & -1 & 0 &\cdots& 0\\
a_1 & a_2 & -\gamma & & \vdots \\
\vdots & \vdots & & \ddots & a_d\\
a_1 & a_2 & \cdots & a_{d-1} & -\gamma
\end{array} \right|    \\                                             
& + \cdots + \left| \begin{array}{ccccc}
-\gamma & a_2 & a_3 &\cdots& a_d  \vspace{.1cm} \\
a_1 & -\gamma & a_3 &\cdots& a_d\\
\vdots & \vdots & \ddots & & \vdots \\
a_1 & a_2 & \cdots & -\gamma & a_d \\
0 & 0 & \cdots & 0 & -1
\end{array} \right|   \\                                               
= & \ -\sum_{j=1}^d (-1)^{d-1} \left[ \gamma^{d-1} - \sum_{\ell=2}^{d-1} (\ell-1) \gamma^{d-1-\ell} \mathop{\sum_{S \subset \{1, \ldots, d\} \setminus \{j\} }}_{\abs{S} = \ell} \ \prod_{i \in S} a_{i} \right] \\
= & \ (-1)^{d} \left[ d \gamma^{d-1}- \sum_{\ell=2}^{d-1} (\ell-1) \gamma^{d-1-\ell} \sum_{j=1}^d \ \mathop{\sum_{S \subset \{1, \ldots, d\} \setminus \{j\} }}_{\abs{S} = \ell} \ \prod_{i \in S} a_{i} \right]  \\
= & \ (-1)^{d} \left[ d \gamma^{d-1}- \sum_{\ell=2}^{d-1} (\ell-1) \gamma^{d-1-\ell} (d-\ell) \ \mathop{\sum_{S \subset \{1, \ldots, d\} }}_{\abs{S} = \ell} \ \prod_{i \in S} a_{i} \right]
\end{align*}
which is the derivative of the right hand side of equation (\ref{polygam}).  The second line follows from the inductive hypothesis.  The last line results from counting the number of values of $j$ such that, given a fixed set $S \subset \{1, \ldots, d\}$ of size $\ell$, $S$ is also a subset of $\{1, \ldots, d \} \setminus \{j\}$ of size $\ell$. This is the same as counting the number of values of $j$ such that $j \notin S$, which is $(d-\ell)$.  Now that we have shown the derivatives to be equivalent, to verify equation (\ref{polygam}) we only need to show that it holds for $\gamma = 0$.  To this end, we can factor the matrix on the left hand side of (\ref{polygam}) when $\gamma = 0$ as:
\begin{align*}
\begin{pmatrix}       
1 & 1 & \cdots & 1  \\
0 & 1 & \ddots& \vdots\\
\vdots &\ddots & \ddots & 1 \\
0 & \cdots  & 0 & 1\\
\end{pmatrix} \begin{pmatrix}        
1 & 0 & \cdots & 0 & 0 \\
0 & 1 & 0 & \cdots & 0\\
\vdots &\ddots & \ddots & \ddots & \vdots \\
0 & \cdots  & 0 & 1 & 0\\
-1 & -2 & \cdots & -(d-1) & 1
\end{pmatrix} & \begin{pmatrix}         
-1 & 1 & 0 & \cdots & 0  \\
0 & -1 & 1 & \ddots & \vdots\\
\vdots &\ddots & \ddots & \ddots & 0 \\
0 & \cdots  & 0 & -1 & 1\\
0 & \cdots & 0 & 0 & (d-1)
\end{pmatrix} \\
& \times \begin{pmatrix}              
a_1 & 0 & \cdots & 0 \\
0 & a_2 & \ddots &  \vdots \\
\vdots &\ddots & \ddots & 0 \\
0 & \cdots  & 0 & a_d\\
\end{pmatrix}.
\end{align*}
Since all of these matricies are triangular, it is clear that the determinant on the left hand side of (\ref{polygam}) when $\gamma = 0$ is $(-1)^{d-1} (d-1) \prod_i a_i$, which is equal to the right hand side of (\ref{polygam}) when $\gamma = 0$.  Thus, we have demonstrated equation (\ref{poly}).
\end{proof}


\begin{proof} [{Proof of Theorem \ref{thmb}}] We begin by considering $C_{\vect{v}}$, the cluster containing a fixed vertex $\vect{v}$, and reveal the vertices in $C_{\vect{v}}$ using a similar algorithm to that described in the proof of Theorem 1, but with one noteworthy modification:  We will remove some additional vertices at step $4$ before they can be observed in any subsequent iterations.  This will avoid the problem of generating closed loops, which would severely reduce the unseen neighborhoods of active vertices and make it difficult to couple the process with a lower-bounding branching process.  The cluster-discovering algorithm is then:
\begin{enumerate}
\item Initialize the sets $R_0 = \emptyset$, $A_0 = \{\vect{v}\}$, $U_0 = V \setminus A_0$.
\item Choose a vertex, $\vect{v}_t$, from $A_t$ in some measurable way (e.g., lexicographically or uniformly at random).
\item $A_{t+1} = A_t \setminus \{\vect{v}_t\} \cup \left\{ \vect{w} \in \mathcal{N}(\vect{v}_t) \cap U_t \ \middle | \ \vect{w} \text{ is occupied} \right\}$
\item $U_{t+1} = \left(U_t \setminus \mathcal{N}(\vect{v}_t) \right) \setminus \left\{ \vect{w} \in \mathcal{N}(\vect{u}) \cap \mathcal{N}(\vect{v}) \ \middle | \ \vect{u}, \vect{v} \in A_{t+1}, \vect{u} \neq \vect{v} \right\} $
\item $R_{t+1} = R_t \cup \{\vect{v}_t\}$
\item If $A_{t+1}$ is empty, then return $R_{t+1}$, otherwise increment $t$ and go to step 2.
\end{enumerate}
We will need to keep track of the number of vertices of each type, so let $\A_t^{(i)}$ denote the number of type $i$ vertices in $A_t$, let $\vect{\A}_t = \left( \A_t^{(1)}, \ldots, \A_t^{(d)} \right)$, and let $\mathcal{N}^i(\vect{v}) = \{ \vect{w} \in V \ | \ \vect{v} - \vect{w} = m \vect{e_i} \text{ for some } m\}$.  As explained in the proof of Theorem~\ref{thma}, a vertex, $\vect{w}$, discovered at time $t$ is labeled type $i$ if $\vect{w} \in \mathcal{N}^i(\vect{v}_t)$.  That is, a vertex is of type $i$ if it is discovered by searching neighboring vertices in the direction of the $i^{\text{th}}$ basis vector.

The above algorithm will give a lower bound on $C_{\vect{v}}$, but we will make use of an upper-bounding branching process, as in the proof of Theorem~\ref{thma}, to tightly control the size of $C_{\vect{v}}$.  We can couple the random walk version of the upper-bounding branching process with $\vect{\A}_t$ by using the same random variables in the construction of $\vect{S}_t$ as in the construction of $\vect{\A}_t$ whenever possible, and when independence is an issue, we just add independent copies of these random variables to $\vect{S}_t$.  More rigorously, let $\xi_{\vect{v}}$ be the random variable that takes the value $1$ if the vertex $\vect{v}$ is occupied and is $0$ otherwise. Then we can define $\vect{S}_t$ iteratively for each $i$:
\begin{equation}
\label{upperBoundCoupling}
{S}_{t+1}^{(i)} = {S}_t^{(i)} + \sum_{\vect{w} \in \mathcal{N}^i (\vect{v}_t) \cap U_t} \xi_{\vect{w}} + \sum_{k=1}^{a_i n - \abs{\mathcal{N}^i(\vect{v}_t) \cap U_t}} \eta_k^{(t,i)} - \mathbbm{1}_{\{\vect{v}_t \text{ is of type } i\}},
\end{equation}
where the $\eta_k^{(t,i)}$ are i.i.d. $\text{Bernoulli}(p)$ random variables, and are independent of the $\xi_{\vect{w}}$.  As long as $\vect{\A}_t \neq \vect{0}$, then we can choose the same focal vertex $\vect{v}_t$ for both processes.  But, if $\vect{\A}_t = \vect{0}$ and $\vect{S}_t \neq \vect{0}$ then we can just choose a vertex from the process $\vect{S}_t$ in some measurable way.

\subsection{Existence and Survival of a Lower-bounding Random Walk} \label{section_supercritical1}

We will make use of this upper bounding random walk in a moment, but first we also need to introduce a lower bounding random walk - also based on a branching process.  We will use the upper bounding random walk to grow a sufficiently large cluster for the lower bounding walk to take over, but first we will need to know what ``sufficiently large'' means.  Also, we need to know that such a lower bound exists, so we begin there.  Ideally, we would like to have a lower bound on $\abs{U_r \cap \mathcal{N}^i (\vect{v})}$ for some large enough $r$.  Using the Erd\"os-R\'enyi model as guidance, we need $n^{(d-1)/2}\ll r \ll n^{d-1}$, so we choose $r = n^{d - \frac{4}{3}}$.

To generate the lower bounding random walk, we will begin by letting $p = p_1 + p_2 - p_1 p_2$, where $p_1 = \frac{\lambda_1}{n}$ and $p = \frac{\lambda}{n}$.  We want to choose $\lambda_1 < \lambda$ so that $\lambda_1 > \lambda_c$, but $\lambda_1$ is smaller than the positive solution to $\det(\vect{M}^{(d-1)}_\lambda - \vect{I}^{(d-1)}) = 0$, where $\vect{M}^{(d-1)}_\lambda$ denotes the submatrix of $\vect{M}_\lambda$ consisting of the first $(d-1)$ rows and columns.  Call this solution $\lambda_c^{(d-1)}$.  In other words, we want $\lambda_1$ to be supercritical for the $d$-dimensional process, but subcritical for the process restricted to any $(d-1)$-dimensional plane.  By the monotonicity of the process of growing a cluster, the critical value for the $(d-1)$-dimensional process obtained by considering the first $(d-1)$ dimensions is smallest, since we assumed that $a_1 \geq a_2 \geq \cdots \geq a_d > 0$.  The critical value for this process is $\lambda_c^{(d-1)}$, so to guarantee the existence of such a $\lambda_1$, we just need to verify that $\lambda_c < \lambda_c^{(d-1)}$.  This is easiest to see by using equation (\ref{poly}).  The values of $\det(\vect{M}_{\lambda_c^{(d-1)}} - \vect{I})$ and $\det(\vect{M}_{0} - \vect{I})$ have opposite signs, so by the Intermediate Value Theorem and the uniqueness of $\lambda_c$, we have that $0 < \lambda_c < \lambda_c^{(d-1)}$.


Consider the process of discovering a cluster starting with $m  = O(\log n)$ vertices (we will decide what $m$ is exactly just before equation (\ref{survival-r})) of arbitrary types in which each unseen vertex has probability $p_1$ of being occupied; call this process $\vect{\A}_t'$. The corresponding active set of vertices will be referred to as $A'_t$, and we have assumed that $\abs{A'_0} = m$.  These $m$ vertices will ultimately come from running the cluster-discovering process up to some time $s = \beta \log n$ starting from a single vertex and using parameter $p$ in Section~\ref{section_supercritical2}.  Conditioned on the survival of this process, we will have that $\abs{R'_0} = \abs{R_s} = s$, where $R'_t$ is the set of occupied vertices which have been removed from the active set up to time $t$ when we start with $m$ active vertices and use parameter $p_1$.  Also, let $U'_t$ be the set of unobserved vertices up to time $t$, and $U'_0 = U_s$.  Denote by 
\begin{equation}
\mathcal{P}_{(i_1, \ldots, i_{k})}^{(\ell_1, \ldots, \ell_{k}) } := \{ \vect{v} \in V \ | \ v_{i_1} = \ell_1, \ldots , v_{i_k} = \ell_k \}
\end{equation}
 the intersection of the vertex set, $V$, with the planes $v_{i_1} = \ell_1, \ldots, v_{i_k} = \ell_{k}$.  Throughout this argument we will assume that $1 \leq i_1, \ldots , i_k \leq d$, $i_{j_1} \neq i_{j_2}$ for $j_1 \neq j_2$, and $\ell_j \in \{ 1, \ldots, a_{i_j} n \}$, so we have no trivial constraints.  We want to bound the number of vertices that become occupied up to time $r$ in any $2$-dimensional plane, $\mathcal{P}_{(i_1, \ldots, i_{d-2})}^{(\ell_1, \ldots, \ell_{d-2}) }$.  This will give us the desired lower bound on  $\abs{U'_r \cap \mathcal{N}^i (\vect{v}')}$ for each $i$ and $\vect{v'}$ that appear in the cluster discovering algorithm.  To get a bound on the number of occupied vertices in these $2$-dimensional planes, we will work backwards by first getting a bound on the number of vertices that become occupied up to time $r$ in $(d-1)$-dimensional planes of the form $\mathcal{P}_{(i_1)}^{(\ell_1)}$.  In the statement of Lemma~\ref{planeBoundLemma}, we allow $m$ and $s$ to be larger than we will need for the proof of Theorem~\ref{thmb}.  This is because we will also use Lemma~\ref{planeBoundLemma} in the proofs of Theorems~\ref{thm-con} and~\ref{thm-isocon}, where $m$ and $s$ will be slightly larger.
 
\begin{lem}
\label{planeBoundLemma}
If $\abs{A'_0} = m = o(n^{2/3})$, $\abs{R'_0} = s = o(n^{2/3})$ and $r = n^{d-4/3}$, then there is a constant $K_d$ such that:
\begin{align}
\nonumber  \prob{\abs{\mathcal{P}_{(i_1,\ldots, i_{d-2})}^{(\ell_1, \ldots, \ell_{d-2})} \cap (A'_r \cup R'_r)} \geq K_d n^{2/3} (\log n)^{d-2} \text{ for some } (i_1, \ldots, i_{d-2})  \text{ and } (\ell_1, \ldots, \ell_{d-2})}{p_1} \\
 \label{pBd}
 =  O(n^{-(d+1)}).
\end{align}
\end{lem}

\begin{proof}[Proof of Lemma \ref{planeBoundLemma}]
We begin with $d=2$.  In this case, equation (\ref{pBd}) is simply the statement that at most $K_2 n^{2/3}$ occupied vertices have been discovered by time $r = n^{2/3}$, so we merely need a large deviation bound.  At each step, we can discover at most a Binomial$(a_1 n, \frac{\lambda_1}{n})$ number of occupied vertices.  Thus,
\begin{align*}
\prob{\abs{A'_r \cup R'_r} \geq 2 a_1 \lambda_1 n^{2/3}}{} &\leq e^{-2 a_1 \lambda_1 r} \E e^{\abs{A'_r \cup R'_r}} \\
& \leq  e^{-2 a_1 \lambda_1 r +m +s} \left[ 1 + \frac{\lambda_1}{n} ( e - 1) \right]^{a_1 n r} \\
& \leq \exp \left[ -2 a_1 \lambda_1 r + a_1 \lambda_1 (e-1) r  +m +s\right] \\
& = O(n^{-3}).
\end{align*}
In the first line above we exponentiated and applied Markov's inequality, in the second line we used that the total number of births up to time $r$ is stochastically bounded by a Binomial$(a_1 n r, \frac{\lambda_1}{n})$ random variable, in the fourth line we used the bound $1+x \leq e^x$, and in the last line we used that $m+s =o(n^{2/3})$ and $e-1 < 2$.  Thus we have (\ref{pBd}) for $d=2$ with $K_2 = 2 a_1 \lambda_1$.

Now, for $d \geq 3$, first consider a fixed plane $\mathcal{P}_{(i_1)}^{(\ell_1)}$.  At each step in the cluster discovering algorithm, either a vertex in this plane is chosen or a vertex in some other parallel plane is chosen.  If a vertex in this plane is chosen, then if we restrict our attention to just this plane, that is we remove $\mathcal{N}^{i_1}(\vect{v}_t')$ from $\mathcal{N}(\vect{v}_t')$ at step 3 of the algorithm, the resulting process will look like a subcritical branching process in $(d-1)$ dimensions, by the assumptions we made on $\lambda_1$.  Thus, restricted only to this plane, no vertex will give rise to a cluster larger than $\beta_1 \log n$ vertices with at least probability $(1 - n^{-d-2})$ (take $\beta_1 = \frac{2d+1}{\alpha_{(d-1)}}$ in the proof of Theorem 1).  If, instead, a vertex in another plane is chosen, then this vertex may have a single opportunity to give rise to an occupied ``seed'' vertex in $\mathcal{P}_{(i_1)}^{(\ell_1)}$.  If the potential seed vertex has already been examined during a previous iteration of the algorithm, then it will not increase the number of seeds in $\mathcal{P}_{(i_1)}^{(\ell_1)}$ whether it is occupied or not.  If the potential seed vertex has not been examined (it is in $U_t$), then it will be occupied with probability $\frac{\lambda_1}{n}$.  So by time $r$, the expected number of seeds is $\leq \frac{\lambda_1 r}{n}$, and each seed is associated with a cluster restricted to $\mathcal{P}_{(i_1)}^{(\ell_1)}$ of size at most $\beta_1 \log n$.  If we let $X_r$ be the number of seed vertices generated in $\mathcal{P}_{(i_1)}^{(\ell_1)}$ by time $r$, then $X_r$ is a sum of a random number of Bernoulli($\frac{\lambda_1}{n}$) random variables.  But $X_r$ is stochastically bounded above by $X'_r \sim$ Binomial$(r, \frac{\lambda_1}{n})$, since the number of potential seed vertices cannot exceed $r$.  Using this fact, and applying a standard generating function argument shows that for $\gamma_1 = 2\log2 - 1 > 0$,
\begin{equation*}
\prob{X_r \geq 2 \frac{\lambda_1 r}{n}}{} \leq e^{-\gamma_1 \frac{\lambda_1 r}{n}}.
\end{equation*}
Recalling that we have chosen $r = n^{d - \frac{4}{3}}$, this implies that
\begin{align*}
\prob{\abs{\mathcal{P}_{(i_1)}^{(\ell_1)} \cap (A'_r \cup R'_r) } \geq 2 \lambda_1 \beta_1 n^{d - 7/3} \log n }{} \leq  e^{-\gamma_1 \lambda_1 n^{d- (7/3)}} + n^{-d-2}
\end{align*}
and since the number of planes, $\mathcal{P}_{(i_1)}^{(\ell_1)}$, is $\leq d a_1 n$, that
\begin{align}
\label{planeBound1}
\prob{\abs{ \mathcal{P}_{(i_1)}^{(\ell_1)} \cap (A'_r \cup R'_r) }  \geq 2 \lambda_1 \beta_1 n^{d - 7/3} \log n \text{ for some } (i_1) \text{ or } (\ell_1) }{} =  O(n^{-(d+1)}).
\end{align}


Now we wish to show that if every plane of the form $\mathcal{P}_{(i_1,\ldots, i_{k-1})}^{(\ell_1, \ldots, \ell_{k-1})}$ has at most $N$ occupied vertices by time $r$, then with high probability, every plane of the form $\mathcal{P}_{(i_1,\ldots, i_{k})}^{(\ell_1, \ldots, \ell_{k})}$ has at most $2 k \lambda_1 \frac{N}{n} \beta_k \log n$ occupied vertices by time $r$.  The reasoning is similar to the argument above for $\mathcal{P}_{(i_1)}^{(\ell_1)}$.  Since no plane of the form $\mathcal{P}_{(i_1,\ldots, i_{k-1})}^{(\ell_1, \ldots, \ell_{k-1})}$ has more than $N$ occupied vertices, then certainly no set of the form $\mathcal{P}_{(i_1,\ldots, i_{k-1})}^{(\ell_1, \ldots, \ell_{k-1})} \setminus \mathcal{P}_{(i_1,\ldots, i_{k})}^{(\ell_1, \ldots, \ell_{k})}$ has more than $N$ occupied vertices.  Let us fix $\mathcal{P}_{(i_1,\ldots, i_{k})}^{(\ell_1, \ldots, \ell_{k})}$, and denote by
$$
\mathcal{P}_{(i_1,\ldots, i_{k}) \setminus (i_j)}^{(\ell_1, \ldots, \ell_{k})\setminus (\ell_j)} := \mathcal{P}_{(i_1,\ldots, i_{j-1}, i_{j+1},\ldots, i_{k})}^{(\ell_1, \ldots,\ell_{j-1}, \ell_{j+1},\ldots,  \ell_{k})} \setminus \mathcal{P}_{(i_1,\ldots, i_{k})}^{(\ell_1, \ldots, \ell_{k})}
$$
In one iteration of the cluster discovering algorithm, the focal vertex, $\vect{v}_t'$, is either in $\mathcal{P}_{(i_1,\ldots, i_{k})}^{(\ell_1, \ldots, \ell_{k})}$, in $\mathcal{P}_{(i_1,\ldots, i_{k}) \setminus (i_j)}^{(\ell_1, \ldots, \ell_{k})\setminus (\ell_j)}$ for some $j = 1, \ldots, k$, or neither.  If it is in neither of these sets, then $\mathcal{N}(\vect{v}_t) \cap \mathcal{P}_{(i_1,\ldots, i_{k})}^{(\ell_1, \ldots, \ell_{k})} = \emptyset$, so this case does not contribute to the number of occupied vertices in this plane.  If the focal vertex is in $\mathcal{P}_{(i_1,\ldots, i_{k})}^{(\ell_1, \ldots, \ell_{k})}$, then we can restrict our attention to just this plane, and the process looks like a subcritical process in $d-k$ dimensions (by our assumptions on $\lambda_1$).  That is, we remove $\cup_{j=1}^k \mathcal{N}^{i_j}(\vect{v}_t')$ from $\mathcal{N}(\vect{v}_t')$ at Step 3 of the cluster discovering algorithm.  Therefore, all vertices in $\mathcal{P}_{(i_1,\ldots, i_{k})}^{(\ell_1, \ldots, \ell_{k})}$ belong to clusters restricted to $\mathcal{P}_{(i_1,\ldots, i_{k})}^{(\ell_1, \ldots, \ell_{k})}$ of size at most $\beta_k \log n$ with probability at least $(1- n^{-(d+k+1)})$ (take $\beta_k = \frac{2d+1 }{\alpha_{(d-k)}}$).  Finally, if the focal vertex is in $\mathcal{P}_{(i_1,\ldots, i_{k}) \setminus (i_j)}^{(\ell_1, \ldots, \ell_{k})\setminus (\ell_j)}$ for some $j = 1, \ldots, k$, then it has precisely one neighbor in $\mathcal{P}_{(i_1,\ldots, i_{k})}^{(\ell_1, \ldots, \ell_{k})}$.  If this neighbor is unseen (is in $U_t$), then it has probability $\frac{\lambda_1}{n}$ of being an occupied seed vertex.  There are at most $kN$ occupied vertices in $\mathcal{P}_{(i_1,\ldots, i_{k}) \setminus (i_j)}^{(\ell_1, \ldots, \ell_{k})\setminus (\ell_j)}$ for $j = 1, \ldots, k$, so the number of seed vertices, $Y_N^k$, is a sum of at most $kN$ Bernoulli$(\frac{\lambda_1}{n})$ random variables, and is thus stochastically bounded above by a Binomial$(kN, \frac{\lambda_1}{n})$ random variable.  Thus, we find that
\begin{equation*}
\prob{Y_N^k \geq 2 \frac{\lambda_1 k N}{n}}{} \leq e^{-\gamma_1 \frac{\lambda_1 k N}{n}},
\end{equation*}
where, again, $\gamma_1 = 2\log2 - 1$.  So, conditional on the event that no plane of the form $\mathcal{P}_{(i_1,\ldots, i_{k-1})}^{(\ell_1, \ldots, \ell_{k-1})}$ has more than $N$ occupied vertices by time $r$, we have that:
\begin{align*}
\prob{\mathcal{P}_{(i_1,\ldots, i_{k})}^{(\ell_1, \ldots, \ell_{k})} \text{ has } \geq 2 k \lambda_1 \beta_k \frac{N}{n} \log n \text{ occupied vertices by time } r}{} \leq  e^{-\gamma_1 \lambda_1 k \frac{N}{n}} + n^{-d-k-1}.
\end{align*}
Since the number of planes, $\mathcal{P}_{(i_1,\ldots, i_{k})}^{(\ell_1, \ldots, \ell_{k})}$, is at most ${d \choose k} a_1^k n^k$, and provided that $N = \omega(n^{1+\epsilon})$ for some $\epsilon >0$, we have:
\begin{align}
\label{planeBound2}
\prob{\text{Any } \mathcal{P}_{(i_1,\ldots, i_{k})}^{(\ell_1, \ldots, \ell_{k})} \text{ has } \geq 2 k \lambda_1 \beta_k \frac{N}{n} \log n \text{ occupied vertices by time } r}{} =  O(n^{-(d+1)})
\end{align}
conditional on the event that no plane of the form $\mathcal{P}_{(i_1,\ldots, i_{k-1})}^{(\ell_1, \ldots, \ell_{k-1})}$ has more than $N$ occupied vertices by time~$r$.

Finally, we can put together equations (\ref{planeBound1}) and (\ref{planeBound2}) to get:
\begin{equation*}
\prob{\text{Any } \mathcal{P}_{(i_1,\ldots, i_{d-2})}^{(\ell_1, \ldots, \ell_{d-2})} \text{ has } \geq K_d n^{2/3} (\log n)^{d-2} \text{ occupied vertices by time } r}{} =  O(n^{-(d+1)})
\end{equation*}
where $K_d = 2^{d-2}\cdot (d-2)! \cdot \lambda_1^{d-2} \cdot \left(\prod_{k=1}^{d-2} \beta_k\right)$
\end{proof}

So with probability at least $1 - O(n^{-(d+1)})$, for $t \leq r$ and $\vect{v}_t$ of type $j \neq i$,
\begin{align*}
\abs{U'_t \cap \mathcal{N}^i (\vect{v}_t')} &\geq a_i n - K_d n^{2/3} (\log n)^{d-2} \\
& \geq a_i (1-\delta) n
\end{align*}
for sufficiently large $n$, where $\delta > 0$ is such that $(1-\delta) \lambda_1 > \lambda_c$.  This implies that, up to time $r$, our process of discovering occupied vertices is bounded below by the random walk version of a multitype branching process with expectation matrix $\vect{M}_{(1-\delta) \lambda_1}$.  Let us call the random walk version of this multitype branching process $\vect{W}_t = (W_t^{(1)}, \ldots, W_t^{(d)})$, and assume that $\vect{W}_0 = \vect{\A}_0'$.  We can couple the processes $\vect{W}_t$ and $\vect{\A}_t'$ until time $T_0 \wedge T_{\delta} \wedge r$, where 
\begin{align*}
T_0 :=& \inf \left\{ t \ \middle | \ \vect{W}_t = \vect{0} \right\}, \\
T_{\delta} :=& \inf \left\{ t \ \middle | \ \abs{U_t \cap \mathcal{N}^i (\vect{v}_t')} < a_i(1-\delta)n \text{ for } \vect{v}_t' \text{ of type } j \text{ and some } i \neq j \right\},
\end{align*}
by first choosing $\vect{v}_t'$ as follows: first decide that $\vect{v}_t'$ will be of type $i$ with probability $\frac{W_{t}^{(i)}}{\norm{\vect{W_{t}}}_1}$, then choose randomly a vertex from $A'_{t}$ that is of this type.  Let $\mathcal{R}_{\delta}^{(i)}(\vect{v}_t')$ be the unseen neighborhood of $\vect{v}_t'$ restricted to the first (lexicographically) $a_i (1-\delta) n$ elements of  $U_t \cap \mathcal{N}^i (\vect{v}_t')$.  Then we can complete the coupling of $\vect{W}_t$ with $\vect{\A}_t'$ as
\begin{align}
\A_{t+1}'^{(i)} &= \A_{t}'^{(i)}  + \sum_{\vect{w} \in U_t \cap \mathcal{N}^i (\vect{v}_t')} \xi_{\vect{w}}' - \mathbbm{1}_{\{\vect{v}_t' \text{ is of type } i \}} \\
W_{t+1}^{(i)} &= W_t^{(i)} + \sum_{\vect{w} \in \mathcal{R}_{\delta}^{(i)} (\vect{v}_t') } \xi_{\vect{w}}' - \mathbbm{1}_{\{\vect{v}_t' \text{ is of type } i \}}
\end{align}
where $\xi_{\vect{w}}'$ is the random variable that takes the value $1$ if vertex $\vect{w}$ is occupied in the random site subgraph obtained using the probability parameter $p_1$, and is $0$ otherwise.


We wish to show now that this coupling will last until time $r$ if $m$ is large enough.  This will occur whenever the branching process corresponding to $\vect{W}_t$ survives.  Let 
$$
q'_{i} = \prob{\vect{W}_t = \vect{0} \text{ for some t } \middle | \ \vect{W}_0 = \vect{e_i} }{}
$$
for $i = 1, \ldots, d$.  Since this branching process is supercritical, $q'_i < 1$ for each $i$.  If we let $e^{-\gamma_2} = \max_i q'_i$ for $\gamma_2 >0$, then
\begin{align*}
 \prob{\vect{W}_t = \vect{0} \text{ for some t } \middle | \ \norm{\vect{W}_0}_1 = m }{} &\leq \left( \max_i q'_i \right)^m \\
 &= e^{-\gamma_2 m}.
\end{align*}
So if $m= \frac{d+1}{\gamma_2} \log n$, then the coupling between $\vect{W}_t$ and $\vect{\mathcal{A}}'_t$ will last until time $r$ with at least probability $1 - n^{-(d+1)}$.  Thus, conditional on $\abs{A'_0} \geq m$, $\abs{A'_0} = o(n^{2/3})$ and $\abs{R'_0} = o( n^{2/3})$ we have
\begin{equation}
\label{survival-r}
\prob{\abs{A'_r} = 0}{p_1} = O\left(n^{-(d+1)}\right).
\end{equation}


\subsection{Establishing a Set of $m$ Active Vertices}
\label{section_supercritical2}
As in the proof of Theorem~\ref{thma}, we make use of an upper-bounding branching process, but this time in a more intimate fashion.  We couple the random walk version of the upper-bounding branching process, $\vect{S}_t$, with $\vect{\mathcal{A}}_t$ as in equation~(\ref{upperBoundCoupling}). Recall from equation (\ref{upperBoundingRW}) that we can write
\begin{equation*}
\vect{S}_{t+1} \deq \vect{S}_t + \sum_{i = 1}^d \mathbbm{1}_{\{j_{t+1} = i\}} \left( \vect{X}_{t+1}^i - \vect{e_i} \right),
\end{equation*}
where $j_{t+1}$ is the random variable that takes the value $i$ if $\vect{v}_t$ is of type $i$ (or, when $\vect{\A}_t = \vect{0}$, if a type $i$ individual is chosen at time $t+1$ in the random walk process), and where the $\vect{X}_{t}^i$ are independent and distributed as $\vect{Z}_1$ conditional on $\vect{Z}_0 = \vect{e}_i$.  For this process (and later for the lower bounding process, $\vect{W}_t$) we will need the following large deviations bounds.


\begin{lem}
\label{largeDeviationsLemma}
Let $\vect{S}_t$ be the random walk version of a branching process with $d$ types in which an individual of type $i$ has a $\text{Binomial}(a_j n, \frac{\lambda}{n})$ number of offspring of type $j \neq i$ and zero offspring of type $i$.  If $\vect{M}_{\lambda} \vect{\mu} = \rho \vect{\mu}$, where $\rho >1$ and $\vect{\mu}$ is the corresponding positive eigenvector normalized so that $\norm{\vect{\mu}}_1 = 1$, then for $y < 1 < x$
\begin{align}
\label{largeDev1}
\prob{\ip{\vect{S}_t - \vect{S}_0}{\vect{\mu}} \geq x (\rho-1) \mu_{\text{max}} \ t }{} \ &\leq e^{-\eta_1 t} \\
\label{largeDev2}
\prob{\ip{\vect{S}_t - \vect{S}_0}{\vect{\mu}} \leq y (\rho-1) \mu_{\text{min}} \ t }{} \ &\leq e^{-\eta_2 t} 
\end{align}
where $\mu_{\text{max}} := \max_i (\mu_i)$, $\mu_{\text{min}} := \min_i (\mu_i)$, and $\eta_1, \eta_2 > 0$ depend on $x$ and $y$, respectively.
\end{lem}

\begin{proof}[Proof of Lemma \ref{largeDeviationsLemma}]
We will first prove inequality (\ref{largeDev1}).  First, as in the proof of Proposition (\ref{propa}), we consider an increment of the process $\ip{\vect{S}_t - \vect{S}_0}{\vect{\mu}}$:
\begin{align*}
\phi_{\vect{v}} (\theta) &:= \E \left[ \exp\left(\theta \ip{\sum_i \mathbbm{1}_{\{j_{t+1} = i\}} \left( \vect{X}_{t+1}^i - \vect{e_i}\right)}{\vect{\mu}} \right) \ \middle | \ \vect{S}_t = \vect{v} \right] \\
&= \sum_i \prob{j_{t+1} = i \ \middle | \ \vect{S}_t = \vect{v} }{} \ \E e^{\theta \ip{(\vect{X}_{t+1}^i - \vect{e}_i)}{\vect{\mu}}} \\
 \label{mgf-bound} &\leq \max_i \ \E e^{\theta \ip{(\vect{X}_1^i - \vect{e}_i)}{\vect{\mu}}}\\
\nonumber &=: \psi(\theta).
\end{align*}
Using this estimate, which is independent of the previous state vector, $\vect{v}$, we see that
\begin{align*}
\E \left[ e^{\theta \ip{\vect{S}_{t+1}}{\vect{\mu}}} \ \middle | \ \vect{S}_t \right] &= \sum_{\vect{v} \geq \vect{0}} \mathbbm{1}_{\{\vect{S}_t = \vect{v} \} } e^{\theta \ip{\vect{v}}{\vect{\mu}}} \ \phi_{\vect{v}}(\theta) \\
&\leq \psi(\theta) \ e^{\theta \ip{\vect{S}_{t}}{\vect{\mu}}}.
\end{align*}
This implies that
\begin{equation}
\label{RWGenFnBd}
\E e^{\theta \ip{\vect{S}_t - \vect{S}_0}{\vect{\mu}}} \leq \left[ \psi(\theta) \right]^t .
\end{equation}
By equation (\ref{mgf-deriv}), for each $i = 1, \ldots, d$
\begin{align*}
\frac{d}{d\theta} \left. \left( \E e^{\theta \ip{(\vect{X}_1^i - \vect{e}_i)}{\vect{\mu}}} \right) \right |_{\theta = 0} &= (\rho - 1) \mu_i \\
&< x (\rho-1) \mu_{\text{max}} \\
&= \frac{d}{d\theta} \left. \left( e^{\theta x (\rho-1) \mu_{\max}} \right) \right|_{\theta = 0}
\end{align*}
since $\rho >1$ and $\vect{\mu} > \vect{0}$.  Because $\psi(0) = 1$, this implies that there exist $\theta_1, \eta_1 > 0$ such that
\begin{equation*}
e^{-\eta_1} := \psi(\theta_1) \ e^{-\theta_1 x (\rho-1) \mu_{\max}} < 1.
\end{equation*}
Thus, by Markov's inequality and inequality (\ref{RWGenFnBd}):
\begin{align*}
\prob{\ip{\vect{S}_t - \vect{S}_0}{\vect{\mu}} \geq x (\rho-1) \mu_{\text{max}} \ t }{} \ &= \prob{e^{\theta_1 \ip{\vect{S}_t - \vect{S}_0}{\vect{\mu}}} \geq e^{\theta_1 x (\rho-1) \mu_{\text{max}} \ t} }{} \\
&\leq \E e^{\theta_1 \ip{\vect{S}_t - \vect{S}_0}{\vect{\mu}}} e^{-\theta_1 x (\rho-1) \mu_{\max}} \\
& \leq e^{-\eta_1 t}.
\end{align*}
So we have proven inequality (\ref{largeDev1}).  To prove inequality (\ref{largeDev2}), we need the following estimate for each $i = 1, \ldots, d$:
\begin{align*}
\frac{d}{d\theta} \left. \left( \E e^{ - \theta \ip{(\vect{X}_1^i - \vect{e}_i)}{\vect{\mu}}} \right) \right |_{\theta = 0} &= - (\rho - 1) \mu_i \\
&< - y (\rho-1) \mu_{\min} \\
&= \frac{d}{d\theta} \left. \left( e^{- \theta y (\rho-1) \mu_{\min}} \right) \right|_{\theta = 0}
\end{align*}
since $\rho > 1$ and $\vect{\mu} > \vect{0}$.  Because $\psi(0) = 1$, this implies that there exist $\theta_2, \eta_2 > 0$ such that 
\begin{equation*}
e^{-\eta_2} := \psi(-\theta_2) \ e^{\theta_2 y (\rho-1) \mu_{\min}} < 1.
\end{equation*}
Thus, by Markov's inequality and inequality (\ref{RWGenFnBd}):
\begin{align*}
\prob{\ip{\vect{S}_t - \vect{S}_0}{\vect{\mu}} \leq y (\rho-1) \mu_{\min} \ t }{} \ &= \prob{e^{-\theta_2 \ip{\vect{S}_t - \vect{S}_0}{\vect{\mu}}} \geq e^{-\theta_2 y (\rho-1) \mu_{\min} \ t} }{} \\
&\leq \E e^{-\theta_2 \ip{\vect{S}_t - \vect{S}_0}{\vect{\mu}}} e^{\theta_2 y (\rho-1) \mu_{\min}} \\
& \leq e^{-\eta_2 t}.
\qedhere
\end{align*}
\end{proof}

We are now ready to prove the following lemma:
\begin{lem}
\label{activeVerticesLemma}
There exists a constant $\beta$ depending on $a_1, \ldots, a_d$ such that if $s = \beta \log n$ then
\begin{equation*}
\prob{0 < \norm{\vect{\A}_s}_1 \leq m}{} = O\left(n^{-(d+1)}\right).
\end{equation*}
and
\begin{equation*}
\prob{\norm{\vect{\A}_s}_1 \geq K \log n}{} = O\left(n^{-(d+1)}\right).
\end{equation*}
where $K \log n >m$.
\end{lem}

\begin{proof}[Proof of Lemma \ref{activeVerticesLemma}]
We wish to apply Lemma \ref{largeDeviationsLemma} to $\vect{S}_t$, but the first step of this process is not the same as the subsequent steps, so we must handle this case separately.  Notice that in the first step of the cluster-discovering algorithm, the initial vertex does not have a type, and is free to search in any of the $d$ directions, while all subsequent vertices are assigned types, and thus expand like a branching process.  Thus, for each $i=1,\ldots, d$, $S_1^{(i)}$ is distributed as a $\text{Binomial}(a_i n, \frac{\lambda}{n})$ random variable.  We can thus bound this step of the process:
\begin{align}
\prob{\ip{\vect{S}_1 - \vect{S}_0}{\vect{\mu}} \geq (d+2) \log n}{} &= \prob{e^{\ip{\vect{S}_1 - \vect{S}_0}{\vect{\mu}}} \geq e^{(d+2) \log n} }{}  \nonumber \\
&\leq n^{-(d+2)} \E e^{\ip{\vect{S}_1 - \vect{S}_0}{\vect{\mu} } } \nonumber \\
&\leq n^{-(d+2)} \exp \left[ \sum_{i=1}^{d} a_i \lambda (e^{\mu_i} - 1) - \mu_{\min} \right] \nonumber \\
\label{firstStepBd}
& = O\left( n^{-(d+2)} \right)
\end{align}
We only need this upper bound, since the first step of $\vect{S}_t$ is stochastically bounded below by any subsequent step, so inequality (\ref{largeDev2}) still holds.  We now apply Lemma \ref{largeDeviationsLemma} to $\vect{S}_t$ at time $s = \beta \log n$ with $x=2$ and $y = \frac{1}{2}$ (and use the inequality in (\ref{firstStepBd}) to handle the first step) to yield
\begin{align*}
\prob{\ip{\vect{S}_s - \vect{S}_0}{\vect{\mu}} \geq \left[2 \beta (\rho-1) \mu_{\max} + (d+2) \right] \ \log n }{} \ &\leq n^{-\eta_1 \beta} + O\left( n^{-(d+2)} \right) \\
\prob{\ip{\vect{S}_s - \vect{S}_0}{\vect{\mu}} \leq \frac{1}{2} \beta (\rho-1) \mu_{\min}\ \log n }{} \ &\leq n^{-\eta_2 \beta}. 
\end{align*}
If we choose $\beta >0$ such that $\eta_1 \beta > d+1$, $\eta_2 \beta > d+1$ and $\beta > \frac{2 \mu_{\max} (d+2)}{\mu_{\min} (\rho-1) \gamma_2}$ (recall that $m = \frac{d+1}{\gamma_2} \log n$), and we let $K = \frac{2 \beta (\rho-1) \mu_{\max} + (d+2)}{\mu_{\min}} + \beta +1 $, then
\begin{align}
\label{observedSetUpperBd}
\prob{\norm{\vect{S}_s}_1 + s \geq K \log n }{} &\leq n^{-(d+1)} \\
\label{activeSetLowerBd}
\prob{\norm{\vect{S}_s}_1 \leq \frac{d+2}{\gamma_2} \log n}{} &\leq n^{-(d+1)}.
\end{align}
To obtain the above inequalities, we used the estimates $\ip{\vect{S}_s - \vect{S}_0}{\vect{\mu}} \geq \mu_{\min} \norm{\vect{S}_s}_1 - \norm{\vect{S}_0}_{\infty}$, $\ip{\vect{S}_s - \vect{S}_0}{\vect{\mu}} \leq \mu_{\max} \norm{\vect{S}_s}_1$ and $\norm{\vect{S}_0}_{\infty} / \mu_{\min} \leq \log n$ for $n$ sufficiently large ($\norm{\vect{S}_0}_{\infty} = 1$).  Equation~(\ref{observedSetUpperBd}) implies the second part of the lemma, since $\norm{\vect{\A}_s}_1 \leq \norm{\vect{S}_s}_1$.

When $\norm{\vect{S}_s}_1 + s \leq K \log n$,  we have:
\begin{equation}
\label{unseenLowerBD}
\abs{U_s \cap \mathcal{N}^i (\vect{v})} \geq a_i n - K \log n
\end{equation}
for $\vect{v}$ of type $j \neq i$.  This is because $\abs{A_s} \leq \norm{\vect{S}_s}_1$, $\abs{R_s} \leq s$ (with equality here if $\A_t$ is still alive at time $s$), and the maximum number of neighbors that $A_s \cup R_s$ can have in $\mathcal{N}^i(\vect{v})$ is $\abs{A_s \cup R_s}$ as long as $\vect{v}$ is not of type $i$.  Therefore, provided $\vect{\A}_s \neq \vect{0}$, the number of births that occur in $\vect{S}_t$ which are lost in $\vect{\A}_t$ for $t \leq s $ is at most
\begin{equation}
\label{extraBirths}
\sum_{t = 0}^{s-1} \ \ \sum_{i=1}^d \ \sum_{k = 1}^{a_i n - \abs{U_t \cap \mathcal{N}^i (\vect{v_t})}} \eta_k^{(t,i)} \leq Y^{(s)}
\end{equation}
where $Y^{(s)} \sim \text{Binomial}\left(d \beta K (\log n)^2, \frac{\lambda}{n} \right)$.  So, conditioned on $\norm{\vect{S}_s}_1 + s \leq K \log n$, we can bound the difference between $\vect{S}_s$ and $\vect{\A}_s$ by a constant:
\begin{align}
\prob{\norm{\vect{\A}_s}_1 >0, \norm{\vect{S}_s - \vect{\A}_s}_1 \geq d+2}{} & \leq \prob{Y^{(s)} \geq d+2}{} \nonumber \\
& = \sum_{k = d+2}^{d \beta K (\log n)^2} {d \beta K (\log n)^2 \choose k} \left(\frac{\lambda}{n}\right)^k \left(1 - \frac{\lambda}{n}\right)^{d \beta K (\log n)^2 - k} \nonumber \\
& \leq \sum_{k=d+2}^{d \beta K (\log n)^2} \left(\frac{\lambda d \beta K (\log n)^2}{n}\right)^{k} \nonumber \\
\label{lostBirthsBd}
&= O\left(n^{-(d+1)}\right).
\end{align}
Finally, combining inequalities (\ref{observedSetUpperBd}), (\ref{activeSetLowerBd}), and (\ref{lostBirthsBd}) yields
\begin{align*}
\prob{0 < \norm{\vect{\A}_s}_1 \leq \frac{d+2}{\gamma_2} \log n - (d+2)}{} = O\left(n^{-(d+1)}\right), \nonumber \\
\prob{0 < \norm{\vect{\A}_s}_1 \leq \frac{d+1}{\gamma_2} \log n }{} = O\left(n^{-(d+1)}\right), \nonumber \\
\prob{0 < \norm{\vect{\A}_s}_1 \leq m}{} = O\left(n^{-(d+1)}\right).
\end{align*}
\end{proof}
This means that, with high probability, at time $s$ the process has either died out or there are at least $m$ active vertices (in $A_s$).


\subsection{Merging Clusters of Size $r = n^{d-4/3}$}
\label{section_supercritical3}
At this point, with high probability, the cluster-discovering algorithm started at a given vertex, $\vect{v}$, has either died out by time $s = \beta \log n$, or will continue to survive until at least time $r = n^{d-4/3}$.  If the process has died out, then the size of the cluster containing $\vect{v}$ is at most $\beta \log n$.  What we now intend to show is that, if the cluster containing the vertex $\vect{v_1}$ and the cluster containing the vertex $\vect{v_2}$ each have at least $r$ vertices, then $\vect{v_1}$ and $\vect{v_2}$ are in the same connected component of the random site subgraph with high probability.

The clusters containing $\vect{v_1}$ and $\vect{v_2}$ will be generated as follows (we use $\vect{v_{1/2}}$ to denote ``$\vect{v_1}$ or $\vect{v_2}$, respectively'', and likewise for similar notation):
\begin{enumerate}
\item Start the cluster-discovering algorithm, $(R_t, A_t, U_t)$, with parameter $p$ at $R_0 = \emptyset$, $A_0 = \vect{v_{1/2}}$ and $U_0 = V \setminus A_0$.  Continue until time $s = \beta \log n$.

\item Start the cluster-discovering algorithm $(R'_t, A'_t, U'_t)$ with parameter $p_1$ at $R'_0 = R_s$, $A'_0 = A_s$ and $U'_0 = U_s$.  Continue until time $r= n^{d - 4/3}$.

\item Let $\Upsilon_{1/2} = R'_{r} \setminus R_s$.
\end{enumerate}

Note that, conditioned on the survival of each process, $\abs{\Upsilon_{1/2}} = r$, and $\Upsilon_{1/2} \subsetneq C_{\vect{v_{1/2}}}$.

\begin{lem}
\label{clusterMergeLemma}
If $\Upsilon_1$ and $\Upsilon_2$ are defined for $\vect{v_1}$ and $\vect{v_2}$ as above and $\abs{\Upsilon_1} = \abs{\Upsilon_2} = r$, then:
\begin{equation*}
\prob{C_{\vect{v_1}} \neq C_{\vect{v_2}}}{p_2} = O(n^{-(d+1)})
\end{equation*}
\end{lem}

\begin{proof}[Proof of Lemma \ref{clusterMergeLemma}]
We will employ a sprinkling technique using the probability reserved in the first step, where the lower bounding process was defined using probability parameter $p_1 = \lambda_1 / n$.  Observe that we can write $p = p_1 + p_2 - p_1 p_2$, where $p = \lambda / n$ and $p_2$ is defined in this way.  The random site subgraph of $G$ with parameter $p$, call it $G_p$, can then be seen as $G_p = \overline{G_{p_1} \cup G_{p_2} }$, where $G_{p_1}$ and $G_{p_2}$ are random site subgraphs of $G$ that are independently generated, the union is taken over their vertex sets, and the bar denotes the inclusion of all edges from $G$ between vertices in $G_{p_1} \cup G_{p_2}$.  We will further subdivide $p_2$ in the same manner, but first it is crucial to note that:
\begin{equation*}
p_2 = \frac{\lambda - \lambda_1}{n} + \frac{\lambda_1 (\lambda - \lambda_1)}{n ( n - \lambda_1)} = \frac{\epsilon}{n} + O(n^{-2}),
\end{equation*}
where $\epsilon = \lambda - \lambda_1$.  We can now write
\begin{equation}
\label{sprinklingProb}
p_2 = \sum_{i=1}^{2d-4} p'_i  - \sum_{1 \leq i< j \leq 2d-4} p'_i p'_j + \cdots + (-1)^{2d-3} \prod_{i=1}^{2d-4} p'_i
\end{equation}
where $p'_i = \epsilon_i / n + O(n^{-2})$ and $\sum_i \epsilon_i = \epsilon$.  This is done merely by repeatedly subdividing $G_p$ into a union of independent subgraphs as described above.  Also, for the sprinkling argument to work as intended (so we can avoid dependencies), we will consider only the vertices in $U_s$, since the vertices in $V \setminus U_s$ have already been fully considered for inclusion in $G_p$.  Since $\abs{U_s \cap \mathcal{N}^i (\vect{v})} \geq a_i n - K \log n$ whenever $\vect{v} \in U_s$ (see inequality (\ref{unseenLowerBD})), this restriction will not significantly affect our estimates below.

We will first treat the cases $d = 2, 3$ separately, then use an induction argument to prove the Lemma for $d \geq 4$.


{\bf Case 1: $d=2$.}  We first observe that with high probability no row or column ($\mathcal{P}_{(i)}^{(\ell)}$ for $i \in \{1, 2\}$ and $\ell \in \{1, \ldots, a_i n \}$) contains more than $\log^2 n$ vertices in $\Upsilon_{1/2}$ because:
\begin{align*}
\prob{\abs{ \mathcal{P}_{(i)}^{(\ell)} \cap \Upsilon_{1/2} } \geq \log^2 n \text{ for some } (i) \text{ and } (\ell) }{}   &\leq \prob{\abs{ \mathcal{P}_{(i)}^{(\ell)} \cap G_p } \geq \log^2 n \text{ for some } (i) \text{ and } (\ell) }{} \\
& \leq 2 a_1 n \left[ e^{-\log^2 n} \sum_{k=1}^{a_1 n}  { {a_1 n} \choose k} p^k e^{k} (1-p)^{a_1 n - k} \right] \\
& \leq 2a_1 n \exp \left[ -\log^2 n + (e-1) a_1 \lambda \right] \\
& = O(n^{1-\log n}).
\end{align*}
If no column ($\mathcal{P}_{(1)}^{(\ell)}$ for $\ell \in \{1, \ldots, a_1 n \}$) has more than $\log^2 n$ vertices that are in $\Upsilon_{1/2}$, then by the Pigeonhole Principle there are at least $r / \log^2 n$ columns with at least one vertex in $\Upsilon_{1/2}$.  Thus we have:
\begin{align}
\label{projectionBound2d}
\prob{\abs{\mathcal{P}_{(1)}^{(\ell)} \cap \mathcal{N}^1 ( \Upsilon_{1/2}) }  \leq \frac{n^{2/3}}{\log^2 n} }{} \leq O(n^{1-\log n}).
\end{align}
We denote by $B_{p_2} (\Upsilon_{1/2})$ the set of vertices in $\mathcal{N}^1(\Upsilon_{1/2})$ that are independently made occupied with parameter $p_2$.  Note that $B_{p_2} (\Upsilon_{1/2})$ contains independent copies of any vertices that have been previously observed.  If there is an edge in $G$ between $\Upsilon_1$ and $\Upsilon_2$ then $\vect{v_1}$ and $\vect{v_2}$ are in the same component and we are finished.  Otherwise, we seek to bound the probability that there is no edge in $G$ between $B_{p_2} (\Upsilon_1)$ and $B_{p_2} (\Upsilon_{2})$.  If there is a column that contains a vertex from each of these two sets then the two sets are connected.  Let $E_{\ell}$ be the event that column $\ell$ ($\mathcal{P}_{(1)}^{(\ell)}$) contains a vertex from $B_{p_2} (\Upsilon_1)$ and $B_{p_2} (\Upsilon_{2})$, and let $E$ be the disjoint union of these events:
\begin{align*}
E_{\ell} &= \{ \mathcal{P}_{(1)}^{(\ell)} \cap B_{p_2} (\Upsilon_1) \neq \emptyset \} \cap \{ \mathcal{P}_{(1)}^{(\ell)} \cap B_{p_2} (\Upsilon_2) \neq \emptyset \}, \\
E &= \bigsqcup_{\ell=1}^{a_1 n} E_{\ell},
\end{align*}
We wish to bound the likelihood of the complementary event:
\begin{align*}
\prob{E_{\ell}^c}{} &\leq 1 - \left[1 - (1- p_2)^{n^{2/3} / \log^2 n} \right]^2 \\
& \leq 1 - \left[ 1 - \exp \left( -\frac{ \epsilon}{n^{1/3}\log^2 n} \right) \right]^2 \\
& \leq 1 - \left[ \frac{ \epsilon}{n^{1/3}\log^2 n}  - \frac{\epsilon^2 }{2 n^{2/3}\log^4 n} \right]^2 \\
& \lae 1 - \frac{\epsilon^2 }{2n^{2/3}\log^4 n} \\
& \leq \exp\left( - \frac{\epsilon^2}{2n^{2/3} \log^4 n} \right).
\end{align*}
In the second and fifth lines above we have applied the bound $1-x \leq e^{-x}$, in the third line we used $1 - e^{-x} \geq x - x^2/2$ for $x \geq 0$, and the fourth line bounds the third for sufficiently large $n$ (which we denote by $\lae$).  Because the events $E_{\ell}$ depend on disjoint columns of vertices, they are independent, so we have:
\begin{align}
\prob{E^c}{} &= \prob{\bigcap_{\ell = 1}^{a_1 n} E_{\ell}^c}{} \nonumber \\
&= \prod_{\ell =1}^{a_1 n} \prob{E_{\ell}^c}{} \nonumber \\
\label{notConnect2d}
& \lae \exp\left( - \frac{a_1 \epsilon^2 n^{1/3} }{2 \log^4 n} \right).
\end{align}
So, combining (\ref{projectionBound2d}) and (\ref{notConnect2d}), for $d = 2$ the probability that $\vect{v_1}$ and $\vect{v_2}$ are not in the same component is bounded above by $O(n^{1-\log n})$.


{\bf Case 2: $d=3$.}   We would like to begin by sprinkling vertices in orthogonal directions from each of the sets $\Upsilon_1$ and $\Upsilon_2$, but we would like to sprinkle in the directions that will give the greatest likelihood of adjoining the two sets.  For this, we need the following fact about subsets of $\mathbb{Z}^3$.

\begin{lem}
\label{projectionLemma}
Let $A \subset \mathbb{Z}^3$, and let $A_x$, $A_y$, and $A_z$, respectively, be projections of $A$ onto the coordinate planes, $yz$-plane, $xz$-plane, and $xy$-plane.  Then 
\begin{equation*}
\abs{A}^{2/3} \leq \abs{A_x}^{1/3} \abs{A_y}^{1/3} \abs{A_z}^{1/3}.
\end{equation*}
\end{lem}

Lemma~\ref{projectionLemma} is a restatement of the main theorem of~\cite{SM:1983} for $d=3$, and it implies that at least one of the projections has size at least $\abs{A}^{2/3}$.  Applying this to the set $\Upsilon_1$, suppose without loss of generality (since the constants $a_i$ will not play a significant part here) that $\abs{\mathcal{P}_{(3)}^{(\ell_3)} \cap \mathcal{N}^3 (\Upsilon_1)} \geq r^{2/3}$ for all $\ell_3$.  Throughout this section, we will also condition on the complement of the event in Lemma~\ref{planeBoundLemma}, which implies that no two-dimensional plane orthogonal to one of the coordinate axes contains more than $n^{2/3} \log n$ vertices in $\Upsilon_2$.  In particular, this implies that $\abs{\mathcal{P}_{(2, 3)}^{(\ell_2, \ell_3)} \cap \Upsilon_2} \leq C n^{2/3} \log n$  for all $(\ell_2, \ell_3)$.  This bounds the number of points that can map to the same point when the set $\Upsilon_2$ is projected onto the span of $\vect{e_2}$ and $\vect{e_3}$, so the Pigeonhole Principle implies that $\abs{ \mathcal{P}_{(1)}^{(\ell_1)} \cap \mathcal{N}^1 (\Upsilon_2) } \geq \frac{n}{C \log n}$.

We begin by sprinkling occupied vertices independently in the neighborhood $\mathcal{N}^3 (\Upsilon_1)$ with probability $p_1' = \epsilon_1/n$ and in the neighborhood $\mathcal{N}^1 (\Upsilon_2)$ with probability $p_2' = \epsilon_2/n$.  We will refer to the sets of newly added vertices (including those which may have already been present in $G_{p_1}$) as $B_{p_1'}^3 (\Upsilon_1)$ and $B_{p_2'}^1 (\Upsilon_2)$, respectively.  The key observation here is that $\vect{v_1}$ and $\vect{v_2}$ will be in the same component if $\mathcal{N}^2(B_{p_1'}^3(\Upsilon_1)) \cap \mathcal{N}^2(B_{p_2'}^1(\Upsilon_2)) \neq \emptyset $.  That is, if there exists a pair $(\ell_1, \ell_3)$ such that $\mathcal{P}_{(1,3)}^{(\ell_1, \ell_3)}$ contains a vertex in $B_{p_1'}^3(\Upsilon_1)$ and a vertex in $B_{p_2'}^1(\Upsilon_2)$, then $\vect{v_1}$ and $\vect{v_2}$ must be in the same component.  Denote by $E_{(\ell_1, \ell_3)}$ the event that $\mathcal{P}_{(1,3)}^{(\ell_1, \ell_3)}$ contains a vertex in $B_{p_1'}^3(\Upsilon_1)$ and a vertex in $B_{p_2'}^1(\Upsilon_2)$, and let $E$ be the disjoint union of these events over all pairs $(\ell_1, \ell_3)$:
\begin{align*}
E_{(\ell_1, \ell_3)} &= \{ \mathcal{P}_{(1,3)}^{(\ell_1, \ell_3)} \cap B_{p_1'} (\Upsilon_1) \neq \emptyset \} \cap \{ \mathcal{P}_{(1,3)}^{(\ell_1, \ell_3)} \cap B_{p_2'} (\Upsilon_2) \neq \emptyset \}, \\
E &= \bigsqcup_{\ell_1=1}^{a_1 n} \bigsqcup_{\ell_3=1}^{a_3 n} E_{(\ell_1, \ell_3)}.
\end{align*}

We now define the following notation:
\begin{align*}
k_{\ell_1}^{(1)} &:= \# \left\{ v_2\in \{1, \ldots, a_2 n\} \, : \, \vect{v} \in \mathcal{P}_{(1)}^{(\ell_1)} \cap \Upsilon_1 \right\}, \\
k_{\ell_3}^{(2)} &:= \# \left\{ v_2\in \{1, \ldots, a_2 n\} \, : \, \vect{v} \in \mathcal{P}_{(3)}^{(\ell_3)} \cap \Upsilon_2 \right\}.
\end{align*}
In other words, $k_{\ell_1}^{(1)}$ is the number of columns, $\mathcal{P}_{(1,2)}^{(\ell_1,\ell_2)} \subset \mathcal{P}_{(1)}^{(\ell_1)}$, which contain at least one vertex in $\Upsilon_1$.  Similarly, $k_{\ell_3}^{(2)}$ is the number of rows, $\mathcal{P}_{(2,3)}^{(\ell_2,\ell_3)} \subset \mathcal{P}_{(3)}^{(\ell_3)}$, which contain at least one vertex in $\Upsilon_2$ (see Figure~\ref{mergeFig3d}).  Note that $k_{\ell_1}^{(1)} = O(n^{2/3} \log n)$ and $k_{\ell_3}^{(2)} = O(n^{2/3} \log n)$, since we are conditioning on the complement of the event in Lemma~\ref{planeBoundLemma}.

\begin{figure}[h]
\begin{center}
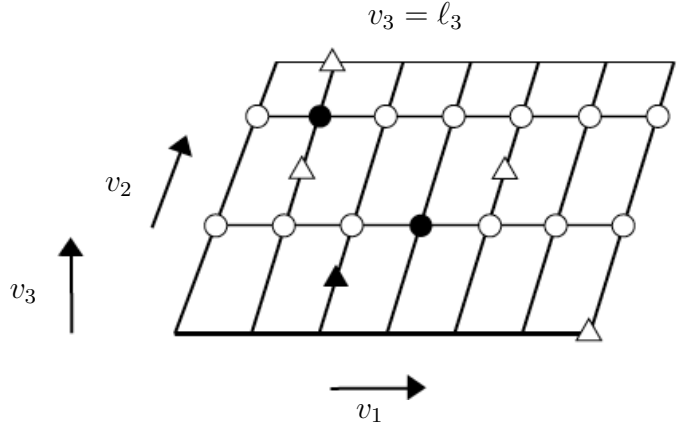
\end{center}
\caption{\label{mergeFig3d}
Solid triangles and circles are vertices in $\Upsilon_1$ and $\Upsilon_2$, respectively.  Empty triangles are neighbors of vertices in $\Upsilon_1$ in the $v_3$ direction, and empty circles are neighbors of vertices in $\Upsilon_2$ in the $v_1$ direction.  Then $k_{\ell_1}^{(1)}$ is the number of triangles (solid or empty) in the column $v_1 = \ell_1$ (e.g. $k_{0}^{(1)} = 0$, $k_{1}^{(1)} = 2$), and $k_{\ell_3}^{(2)}$ is the number of rows with circles in the plane $v_3 = \ell_3$ (here $k_{\ell_3}^{(2)}=2$).  Also, $E_{(\ell_1, \ell_3)}$ is the event that in the column $v_1 = \ell_1$, at least one triangle (solid or empty) is made occupied with parameter $p'_1$ and at least one circle (solid or empty) is made occupied with parameter~$p'_2$. }
\end{figure}

We can now write the probabilities of $E_{(\ell_1, \ell_3)}^c$ and $E^c$, conditional on the values of $k_{\ell_1}^{(1)}$ and $k_{\ell_3}^{(2)}$, as:
\begin{align*}
\prob{E_{(\ell_1, \ell_3)}^c}{} &= 1 - \left(1 - \left(1-\frac{\epsilon_1}{n} \right)^{k_{\ell_1}^{(1)}} \right) \left(1 - \left(1-\frac{\epsilon_2}{n} \right)^{k_{\ell_3}^{(2)}} \right), \\
\prob{E^c}{} &= \prod_{\ell_1 = 1}^{a_1 n}  \prod_{\ell_3 = 1}^{a_3 n} \left[1 - \left(1 - \left(1-\frac{\epsilon_1}{n} \right)^{k_{\ell_1}^{(1)}} \right) \left(1 - \left(1-\frac{\epsilon_2}{n} \right)^{k_{\ell_3}^{(2)}} \right) \right],
\end{align*}
and we wish to bound these quantities.  To this end, we observe that for large $n$:
\begin{align}
\nonumber 1 - \left(1 - \frac{\epsilon_{1/2}}{n} \right)^k &=  \frac{k \epsilon_{1/2}}{n} - \sum_{i=2}^k {k \choose i} \left(\frac{\epsilon_{1/2}}{n} \right)^i (-1)^i \\
\nonumber &\geq  \frac{k \epsilon_{1/2}}{n} - \left( \frac{k\epsilon_{1/2}}{n} \right)^2 \sum_{i = 0}^{k} \left(\frac{k \epsilon_{1/2}}{n} \right)^i \\
\label{asymptoticBound}
& \gae  \frac{k \epsilon_{1/2}}{2n},
\end{align}
provided $k = o(n)$, which is the case for $k_{\ell_1}^{(1)}$ and $k_{\ell_3}^{(2)}$.  Applying this inequality and $1-x \leq e^{-x}$ gives:
\begin{align}
\nonumber \prob{E^c}{} &\lae \prod_{\ell_1 = 1}^{a_1 n}  \prod_{\ell_3 = 1}^{a_3 n} \left[ 1 - \left(\frac{k_{\ell_1}^{(1)} \epsilon_1}{2n} \right) \left(\frac{k_{\ell_3}^{(2)} \epsilon_2}{2n} \right) \right] \\
\label{intersectBd1}
&\leq \exp \left[ - \frac{\epsilon_1 \epsilon_2}{4n^2} \sum_{\ell_1 = 1}^{a_1 n} k_{\ell_1}^{(1)} \sum_{\ell_3 = 1}^{a_3 n} k_{\ell_3}^{(2)} \right].
\end{align}
What remains now is to obtain lower bounds for the sums in the above expression.  It follows from the definitions of $k_{\ell_1}^{(1)}$ and $k_{\ell_3}^{(2)}$ that for any $\ell_3$
\begin{align*}
\sum_{\ell_1=1}^{a_1 n} k_{\ell_1}^{(1)} = \abs{ \mathcal{P}_{(3)}^{(\ell_3)} \cap \mathcal{N}^3 (\Upsilon_1) } \geq r^{2/3} = n^{10/9},
\end{align*}
and for any $\ell_1$
\begin{align*}
\sum_{\ell_3=1}^{a_3 n} k_{\ell_3}^{(2)} = \abs{ \mathcal{P}_{(1)}^{(\ell_1)} \cap \mathcal{N}^1 (\Upsilon_2) } \geq \frac{n}{C \log n}.
\end{align*}
Applying the last two lower bounds to (\ref{intersectBd1}), then handling the exceptional case where some plane contains more than $C n^{2/3} \log n$ vertices in $\Upsilon_{1/2}$ by using Lemma~\ref{planeBoundLemma}, yields
\begin{align*}
\prob{E^c}{} &\lae \exp \left[ - \frac{\epsilon_1 \epsilon_2}{4 n^2} \left( \frac{n}{C \log n} \right) \left(n^{10/9}\right) \right] + O \left( n^{-(d+1)} \right)\\
& = O \left( n^{-(d+1)} \right).
\end{align*}


{\bf Case 3: $d\geq4$.}  We will ultimately use the same argument here as in Case 2, but first we will sprinkle vertices in the directions $\vect{e_d}, \vect{e_{d-1}}, \ldots, \vect{e_4}$ sequentially.  At that point we will have $\prod_{i = 4}^d a_i \ n^{d-3}$ independent $3$-dimensional Hamming Tori in which the two sets can connect.  So that we can independently connect vertices to the sets $\Upsilon_{1/2}$, we will write $p_2$ as in equation~(\ref{sprinklingProb}) with $k = 2d-4$.  We will denote by $B^i_{p}(J)$ the set of vertices that are independently made occupied with probability $p$ in the neighborhood $\mathcal{N}^i(J)$.

We now begin by sprinkling vertices in the $\vect{e_d}$ direction to generate the sets $B_{p'_{1/2}}^d (\Upsilon_{1/2})$.  By conditioning on the complement of the event in Lemma~\ref{planeBoundLemma}, we have that for every $(\ell_1, \ldots, \ell_{d-1})$
\begin{equation*}
\abs{ \mathcal{P}_{(1, \ldots, d-1)}^{(\ell_1, \ldots, \ell_{d-1})} \cap \Upsilon_{1/2} } \leq C n^{2/3} (\log n)^{d-2} .
\end{equation*}
This bound says that when the sets $\Upsilon_{1/2}$ are projected onto the span of $\vect{e_1}, \ldots, \vect{e_{d-1}}$,  no more than this many vertices can be mapped onto the same point.  Thus, by the Pigeonhole Principle, we have a lower bound on the size of the projection:
\begin{equation*}
\abs{ \mathcal{P}_{(d)}^{(\ell_d)} \cap \mathcal{N}^d(\Upsilon_{1/2}) } \geq \frac{r}{C n^{2/3} (\log n)^{d-2}} = \frac{n^{d-2}}{C (\log n)^{d-2}}
\end{equation*}
for any $\ell_d$.  Conditional on this bound, we obtain the following bound on the the number of newly occupied vertices in $B_{p'_{1/2}}^d (\Upsilon_{1/2})$ in a fixed $(d-1)$-dimensional Hamming torus:
\begin{align*}
\prob{\abs{ \mathcal{P}_{(d)}^{(\ell_d)} \cap B_{p'_{1/2}}^d(\Upsilon_{1/2}) } \leq \frac{n^{d-3}}{(\log n)^{d-1}} }{}           &   \leq \exp\left[ \frac{n^{d-3}}{(\log n)^{d-1}}\right]                \cdot             \E \exp \left[ - \abs{ \mathcal{P}_{(d)}^{(\ell_d)} \cap B_{p'_{1/2}}^d(\Upsilon_{1/2}) } \right] \\
& = \exp\left[ \frac{n^{d-3}}{(\log n)^{d-1}}\right]  \cdot    \left[ 1 - \frac{\epsilon_{1/2}}{n} (1 - e^{-1}) \right]^{\abs{ \mathcal{P}_{(d)}^{(\ell_d)} \cap \mathcal{N}^d(\Upsilon_{1/2}) } }\\
&\leq \exp\left[\frac{n^{d-3}}{(\log n)^{d-1}} - \frac{\epsilon_{1/2}\, n^{d-3}}{C (\log n)^{d-2}} (1 - e^{-1}) \right] \\
&\lae \exp \left[ -\eta'_3 \frac{n^{d-3}}{(\log n)^{d-2}} \right]
\end{align*}
for some constant $\eta'_3 >0$.  In the first line above we exponentiated and applied Markov's inequality.  In the second line we used the moment generating function for a binomial random variable, and in the third line we used the bound on the size of the projection that we obtained above and the inequality $1-x \leq e^{-x}$.  By the union bound, we have the following bound on the number of newly occupied vertices in each $(d-1)$-dimensional Hamming Torus:
\begin{align}
\label{firstProjection}
\nonumber \prob{\exists \ \ell_d \ \text{s.t.} \  \abs{ \mathcal{P}_{(d)}^{(\ell_d)} \cap B_{p'_{1/2}}^d(\Upsilon_{1/2}) } \leq \frac{n^{d-3}}{(\log n)^{d-1}}}{}    &\leq a_d n \cdot  \exp \left[ -\eta'_3 \frac{n^{d-3}}{(\log n)^{d-2}} \right] \\
&\lae  \exp \left[ -\eta_3 \frac{n^{d-3}}{(\log n)^{d-2}} \right]
\end{align}
for some constant $\eta_3 >0$.  For $d = 4$ this is as far as we will need to project before applying the argument from Case 2.  For that argument to work, and so we can continue to project down to lower dimensions for $d>4$, we will need an estimate for the maximum number of vertices in  $\mathcal{P}_{(d)}^{(\ell_d)} \cap B_{p'_{1/2}}^d(\Upsilon_{1/2})$ that can map to a single vertex when the set is projected onto a plane of one less dimension.  This will be easy by again using Markov's inequality and the moment generating function for a binomial random variable.  For any $1 \leq i_1 < i_2 < \cdots < i_{d-2} < d$:
\begin{align}
\label{firstProjectionLB}
\nonumber &\prob{ \exists \ (\ell_{i_1}, \ldots, \ell_{i_{d-2}}, \ell_d) \ \text{s.t.} \ \abs{ \mathcal{P}_{(i_1, \ldots, i_{d-2}, d)}^{(\ell_{i_1}, \ldots, \ell_{i_{d-2}}, \ell_d)} \cap B^d_{p'_{1/2}}(\Upsilon_{1/2}) } \geq \log^2 n }{} \\
\nonumber &\leq \prod_i a_i \ n^{d-1} \ \exp\left[ - \log^2 n \right] \cdot \E \exp\left[ \abs{ \mathcal{P}_{(i_1, \ldots, i_{d-2}, d)}^{(\ell_{i_1}, \ldots, \ell_{i_{d-2}}, \ell_d)} \cap B^d_{p'_{1/2}}(\Upsilon_{1/2}) } \right] \\
\nonumber &= \prod_i a_i \ n^{d-1} \ \exp\left[ - \log^2 n \right] \cdot \left [ 1 + \frac{\epsilon_{1/2}}{n} (e - 1) \right]^{\abs{ \mathcal{P}_{(i_1, \ldots, i_{d-2}, d)}^{(\ell_{i_1}, \ldots, \ell_{i_{d-2}}, \ell_d)} \cap \mathcal{N}^d(\Upsilon_{1/2})} } \\ 
\nonumber &\leq \prod_i a_i \ n^{d-1} \ \exp\left[ - \log^2 n + \frac{\epsilon_{1/2}\, C\, n^{2/3} (\log n)^{d-2} }{n} (e - 1) \right] \\
& \lae \exp\left[-\eta_4 \, \log^2 n \right].
\end{align}
The third and fourth lines above are conditional on the complement of the event in Lemma~\ref{planeBoundLemma}, and assume that the sets $\Upsilon_{1/2}$ have been determined accordingly.

Now, if $d > 4$, we will repeat the above argument until we are within $3$-dimensional Hamming tori.  Suppose, for some $k$ ($1\leq k \leq d-4$), that
\begin{equation}
\label{projectionAssumption1}
\abs{ \mathcal{P}_{(d-k+1, \ldots, d)}^{(\ell_{d-k+1},\ldots, \ell_d)} \cap B^{d-k+1}_{p'_{(2k-1)/(2k)}}(B^{d-k+2}_{p'_{(2k-3)/(2k-2)}}( \cdots (B^d_{p'_{1/2}}(\Upsilon_{1/2})\, ) \cdots )} \geq N
\end{equation}
for all $(\ell_{d-k+1},\ldots, \ell_d)$ where $N = \omega(n \log^3 n)$, and that
\begin{equation}
\label{projectionAssumption2}
\abs{ \mathcal{P}_{(i_1, \ldots,i_{d-k-1}, d-k+1, \ldots, d)}^{(\ell_{i_1}, \ldots, \ell_{i_{d-k-1}}, \ell_{d-k+1},\ldots, \ell_d)} \cap B^{d-k+1}_{p'_{(2k-1)/(2k)}}(B^{d-k+2}_{p'_{(2k-3)/(2k-2)}}( \cdots (B^d_{p'_{1/2}}(\Upsilon_{1/2})\, ) \cdots )} \leq \log^2 n
\end{equation}
for all $1 \leq i_1 < \cdots < i_{d-k-1} < d-k+1$ and $(\ell_{i_1}, \ldots, \ell_{i_{d-k-1}}, \ell_{d-k+1},\ldots, \ell_d)$.
The first inequality above says that the sets $\Upsilon_{1/2}$ each have at least $N$ (most recently occupied) neighbors in each of the $\Theta(n^{k})$ $(d-k)$-dimensional sub-Hamming tori after newly occupied vertices have been sprinkled successively in the directions $\vect{e_d}, \ldots, \vect{e_{d-k+1}}$.  The second inequality says that no line contains more than $\log^2 n$ of the most recently added occupied vertices.  Under these two assumptions, and by the Pigeonhole Principle, we have that the size of the projection onto the span of $\vect{e_1}, \ldots, \vect{e_{d-k-1}}$ is bounded by:
\begin{equation*}
\abs{ \mathcal{P}_{(d-k, \ldots, d)}^{(\ell_{d-k},\ldots, \ell_d)} \cap \mathcal{N}^{d-k}( B^{d-k+1}_{p'_{(2k-1)/(2k)}}( \cdots (B^d_{p'_{1/2}}(\Upsilon_{1/2})\, ) \cdots )} \geq \frac{N}{\log^2 n}
\end{equation*}
for any $(\ell_{d-k},\ldots, \ell_d)$.  Using Markov's inequality and the moment generating function for a binomial random variable as we did above, we have that
\begin{align}
\nonumber &\prob{\exists \ (\ell_{d-k},\ldots, \ell_d) \ \text{s.t.} \ \abs{ \mathcal{P}_{(d-k, \ldots, d)}^{(\ell_{d-k},\ldots, \ell_d)} \cap B^{d-k}_{p'_{(2k+1)/(2k+2)}}( B^{d-k+1}_{p'_{(2k-1)/(2k)}}( \cdots (B^d_{p'_{1/2}}(\Upsilon_{1/2})\, ) \cdots )} \leq \frac{N}{n \log^3 n} }{} \\
\label{inductiveProjection}
&\hspace{6cm} \lae \exp \left[ - \eta_{2k+1} \frac{N}{n \log^2 n} \right]
\end{align}
for some constant $\eta_{2k+1} >0 $.  Similarly, for all  $1 \leq i_1 < \cdots < i_{d-k-1} < d-k$ and all $\vect{\ell} = (\ell_{i_1}, \ldots, \ell_{i_{d-k-2}}, \ell_{d-k},\ldots, \ell_d)$,
\begin{align*}
\abs{ \mathcal{P}_{(i_1, \ldots,i_{d-k-2}, d-k, \ldots, d)}^{(\ell_{i_1}, \ldots, \ell_{i_{d-k-2}}, \ell_{d-k},\ldots, \ell_d)} \cap \mathcal{N}^{d-k}( B^{d-k+1}_{p'_{(2k-1)/(2k)}}(\cdots (B^d_{p'_{1/2}}(\Upsilon_{1/2})\, ) \cdots )} \leq \log^2 n,
\end{align*}
which implies that
\begin{align}
\label{inductiveProjectionUB}
\nonumber &\prob{ \exists \ \vect{\ell} \ \text{s.t.} \ \abs{ \mathcal{P}_{(i_1, \ldots,i_{d-k-2}, d-k, \ldots, d)}^{(\ell_{i_1}, \ldots, \ell_{i_{d-k-2}}, \ell_{d-k},\ldots, \ell_d)} \cap B^{d-k}_{p'_{(2k+1)/(2k+2)}}( B^{d-k+1}_{p'_{(2k-1)/(2k)}}(\cdots (B^d_{p'_{1/2}}(\Upsilon_{1/2})\, ) \cdots )} \geq \log^2 n}{} \\
& \hspace{6cm} \lae \exp \left[ -\eta_{2k+2} \log^2 n \right].
\end{align}
Putting together inequalities (\ref{firstProjection}) through (\ref{inductiveProjectionUB}) inductively gives (for $d\geq 4$): 
\begin{align}
\label{projectionLB}
\nonumber &\prob{ \exists \ \vect{\ell} \ \text{s.t.} \ \abs{ \mathcal{P}_{(4, \ldots, d)}^{(\ell_{4}, \ldots, \ell_d)} \cap B^{4}_{p'_{(2d-7)/(2d-6)}}( B^{5}_{p'_{(2d-9)/(2d-8)}}(\cdots (B^d_{p'_{1/2}}(\Upsilon_{1/2})\, ) \cdots )} \leq \frac{n}{(\log n)^{4d-13}}}{} \\
& \hspace{5.5cm} \lae \exp \left[ -\eta_{2d-5} \log^2 n \right].
\end{align}
Here, $\vect{\ell} = (\ell_{4},\ldots, \ell_d)$, and henceforth $\vect{\ell}$ will be the vector in the superscript of the next ``$\mathcal{P}$''. We used the union bound to control the probability above, and the constant $\eta_{2d-5}>0$ is chosen such that the sum of the probabilities of the events involved is bounded asymptotically by the above expression. These events are still conditional on complement of the event in Lemma~\ref{planeBoundLemma}.  Likewise, we have:
\begin{align}
\label{projectionUB1}
\nonumber &\prob{ \exists \ \vect{\ell} \ \text{s.t.} \ \abs{ \mathcal{P}_{(1, 2, 4, \ldots, d)}^{(\ell_1, \ell_2, \ell_{4}, \ldots, \ell_d)} \cap B^{4}_{p'_{(2d-7)}}( B^{5}_{p'_{(2d-9)}}(\cdots (B^d_{p'_{1}}(\Upsilon_{1})\, ) \cdots )} \geq \log^2 n}{} \\
& \hspace{4.5cm} \lae \exp \left[ -\eta_{2d-4} \log^2 n \right],
\end{align}
and
\begin{align}
\label{projectionUB2}
\nonumber &\prob{ \exists \ \vect{\ell} \ \text{s.t.} \ \abs{ \mathcal{P}_{(2, 3, 4, \ldots, d)}^{(\ell_2, \ell_3, \ell_{4}, \ldots, \ell_d)} \cap B^{4}_{p'_{(2d-6)}}( B^{5}_{p'_{(2d-8)}}(\cdots (B^d_{p'_{2}}(\Upsilon_{2})\, ) \cdots )} \geq \log^2 n}{} \\
& \hspace{4.5cm} \lae \exp \left[ -\eta_{2d-3} \log^2 n \right].
\end{align}

Within each $3$-dimensional Hamming torus, $\mathcal{P}_{(4, \ldots, d)}^{(\ell_{4}, \ldots, \ell_d)}$, we can now apply the argument from Case $2$.  This time we will not need Lemma~\ref{projectionLemma} because we have a better bound on the number of (most recently) occupied vertices in any line.  Fixing $(\ell_4, \ldots, \ell_d)$, if we project the vertices most recently added to $\Upsilon_1$ onto the plane spanned by $\vect{e_1}$ and $\vect{e_2}$, then inequalities (\ref{projectionLB}) and (\ref{projectionUB1}) imply (by the Pigeonhole Principle) that with high probability
\begin{align}
\label{projectionPigeonhole1}
\abs{ \mathcal{P}_{(3, \ldots, d)}^{(\ell_{3}, \ldots, \ell_d)} \cap \mathcal{N}^3 (B^{4}_{p'_{(2d-7)}}(\cdots (B^d_{p'_{1}}(\Upsilon_{1})\, ) \cdots )} \geq \frac{n}{(\log n)^{4d-11}}.
\end{align}
Likewise, if we project the vertices added most recently to $\Upsilon_2$ onto the plane spanned by $\vect{e_2}$ and $\vect{e_3}$ and apply the Pigeonhole Principle with inequalities (\ref{projectionLB}) and (\ref{projectionUB2}), then the size of the projection within each $3$-dimensional sub-Hamming torus is at least:
\begin{align}
\label{projectionPigeonhole2}
\abs{ \mathcal{P}_{(1, 4, \ldots, d)}^{(\ell_{1}, \ell_{4}, \ldots, \ell_d)} \cap \mathcal{N}^1 (B^{4}_{p'_{(2d-6)}}(\cdots (B^d_{p'_{2}}(\Upsilon_{2})\, ) \cdots )} \geq \frac{n}{(\log n)^{4d-11}}.
\end{align}

Now, as in Case 2, we will condition on the last two inequalities, and sprinkle new vertices independently in the $\vect{e_3}$ direction from $B^{4}_{p'_{(2d-7)}}(\cdots (B^d_{p'_{1}}(\Upsilon_{1})\, ) \cdots )$ and in the $\vect{e_1}$ direction from $B^{4}_{p'_{(2d-6)}}(\cdots (B^d_{p'_{2}}(\Upsilon_{2})\, ) \cdots )$.  The key observation here is that $\vect{v_1}$ and $\vect{v_2}$ will be in the same component if
\begin{equation*}
\mathcal{N}^2 (B^{3}_{p'_{(2d-5)}}(\cdots (B^d_{p'_{1}}(\Upsilon_{1})\, ) \cdots )\ \bigcap  \ \mathcal{N}^2 (B^{1}_{p'_{(2d-4)}}(\cdots (B^d_{p'_{2}}(\Upsilon_{2})\, ) \cdots ) \neq \emptyset.
\end{equation*}
That is, if there exists a vector $(\ell_1, \ell_3, \ell_4, \ldots, \ell_d)$ such that $\mathcal{P}^{(\ell_1, \ell_3, \ell_4, \ldots, \ell_d)}_{(1, 3, 4, \ldots, d)}$ contains a vertex in \\ $B^{3}_{p'_{(2d-5)}}(\cdots (B^d_{p'_{1}}(\Upsilon_{1})\, ) \cdots )$ and a vertex in $B^{1}_{p'_{(2d-4)}}(\cdots (B^d_{p'_{2}}(\Upsilon_{2})\, ) \cdots )$.  If we define the following events
\begin{equation*}
E_{(\ell_1, \ell_3, \ell_4, \ldots, \ell_d)} := \left\{ B^{3}_{p'_{(2d-5)}}(\cdots (B^d_{p'_{1}}(\Upsilon_{1})\, ) \cdots ) \, \cap \, B^{1}_{p'_{(2d-4)}}(\cdots (B^d_{p'_{2}}(\Upsilon_{2})\, ) \cdots ) \, \cap \, \mathcal{P}^{(\ell_1, \ell_3, \ell_4, \ldots, \ell_d)}_{(1, 3, 4, \ldots, d)} \neq \emptyset \right\}
\end{equation*}
\begin{equation*}
E_{(\ell_4, \ldots, \ell_d)} := \bigsqcup_{\ell_1=1}^{a_1 n} \bigsqcup_{\ell_3=1}^{a_3 n} E_{(\ell_1, \ell_3, \ell_4, \ldots, \ell_d)}
\end{equation*}
\begin{equation*}
E := \bigsqcup_{\ell_4=1}^{a_4 n} \cdots \bigsqcup_{\ell_d=1}^{a_d n} E_{(\ell_4, \ldots, \ell_d)} 
\end{equation*}
then $E$ is contained in the event that $\vect{v_1}$ and $\vect{v_2}$ are in the same component.  To estimate the probabilities of these events, let us temporarily fix $(\ell_4, \ldots, \ell_d)$, and define the variables:
\begin{align*}
k_{\ell_1}^{(1)} &:= \# \left\{ v_2\in \{1, \ldots, a_2 n\} \, : \, \vect{v} \in \mathcal{P}_{(1, 4, \ldots, d)}^{(\ell_1, \ell_4, \ldots, \ell_d)} \cap B^{4}_{p'_{(2d-7)}}(\cdots (B^d_{p'_{1}}(\Upsilon_{1})\, ) \cdots ) \right\}, \\
k_{\ell_3}^{(2)} &:= \# \left\{ v_2\in \{1, \ldots, a_2 n\} \, : \, \vect{v} \in \mathcal{P}_{(3, 4, \ldots, d)}^{(\ell_3, \ell_4, \ldots, \ell_d))} \cap B^{4}_{p'_{(2d-6)}}(\cdots (B^d_{p'_{2}}(\Upsilon_{2})\, ) \cdots ) \right\}.
\end{align*}
In other words, $k_{\ell_1}^{(1)}$ is the number of columns, $\mathcal{P}_{(1, 2, 4, \ldots, d)}^{(\ell_1, \ell_2, \ell_4, \ldots, \ell_d)} \subset \mathcal{P}_{(1, 4, \ldots, d)}^{(\ell_1, \ell_4, \ldots, \ell_d)}$, which contain at least one vertex in $B^{4}_{p'_{(2d-7)}}(\cdots (B^d_{p'_{1}}(\Upsilon_{1})\, ) \cdots )$.  Similarly, $k_{\ell_3}^{(2)}$ is the number of rows, $\mathcal{P}_{(2, 3, 4, \ldots, d)}^{(\ell_2, \ell_3, \ell_4, \ldots, \ell_d)} \subset \mathcal{P}_{(3, 4, \ldots, d)}^{(\ell_3, \ell_4, \ldots, \ell_d)}$, which contain at least one vertex in $B^{4}_{p'_{(2d-6)}}(\cdots (B^d_{p'_{2}}(\Upsilon_{2})\, ) \cdots )$.  We can now estimate the probability of $E^c$, conditional on the complements of the events in inequalities (\ref{projectionLB}) - (\ref{projectionUB2}).
\begin{align}
\nonumber \prob{E_{(\ell_1, \ell_3, \ell_4, \ldots, \ell_d)}^c }{} &= \ 1 - \left(1 - \left(1-p'_{(2d-5)}\right)^{k_{\ell_1}^{(1)}}  \right) \left(1 - \left(1-p'_{(2d-4)} \right)^{k_{\ell_3}^{(2)}} \right) \\
\nonumber \prob{E_{(\ell_4, \ldots, \ell_d)}^c}{} &\lae \prod_{\ell_1 = 1}^{a_1 n} \prod_{\ell_3 = 1}^{a_3 n} \left[ 1 - \left(1 - \left(1-\frac{\epsilon_{(2d-5)}}{2n} \right)^{k_{\ell_1}^{(1)}}  \right) \left(1 - \left(1-\frac{\epsilon_{(2d-4)}}{2n} \right)^{k_{\ell_3}^{(2)}} \right) \right] \\
\nonumber &\lae \prod_{\ell_1 = 1}^{a_1 n} \prod_{\ell_3 = 1}^{a_3 n} \left[ 1 - \frac{\epsilon_{(2d-5)} \epsilon_{(2d-4)}}{16n^2} k_{\ell_1}^{(1)} k_{\ell_3}^{(2)} \right] \\
\nonumber &\leq \exp \left[ - \frac{\epsilon_{(2d-5)} \epsilon_{(2d-4)}}{16n^2} \sum_{\ell_1 = 1}^{a_1 n} k_{\ell_1}^{(1)} \sum_{\ell_3 = 1}^{a_3 n}  k_{\ell_3}^{(2)} \right]\\
\nonumber &\leq \exp \left[ - \frac{\epsilon_{(2d-5)} \epsilon_{(2d-4)}}{16n^2} \left( \frac{n}{(\log n)^{4d-11}} \right)^2 \right] \\
\nonumber &= \exp \left[ - \frac{\epsilon_{(2d-5)} \epsilon_{(2d-4)}}{16} \cdot \frac{1}{(\log n)^{8d - 22}} \right] \\
\label{clusterMergeBound}
\prob{E^c}{} &\lae \exp \left[- \frac{\epsilon_{(2d-5)} \epsilon_{(2d-4)}}{16} \cdot \frac{\prod_{i=4}^d a_i \ n^{d-3}}{(\log n)^{8d - 22}} \right]
\end{align}
In the second line of the above inequalities, the probability on the left is asymptotically less than the quantity on the right because $p'_i = \epsilon_i / n + O(n^2)$.  The third line is obtained from inequality (\ref{asymptoticBound}), which holds here because $k_{\ell_1}^{(1)}, k_{\ell_3}^{(2)} = o(n)$ with high probability from inequalities (\ref{projectionUB1}) and (\ref{projectionUB2}).  The fourth line is an application of the bound $1-x \leq e^{-x}$.  The fifth line is from the definitions of $k_{\ell_1}^{(1)}$ and $k_{\ell_3}^{(2)}$, the observation that
\begin{align*}
\sum_{\ell_1 = 1}^{a_1 n} k_{\ell_1}^{(1)} &= \abs{ \mathcal{P}_{(3, 4, \ldots, d)}^{(\ell_3, \ell_{4}, \ldots, \ell_d)} \cap \mathcal{N}^3(B^{4}_{p'_{(2d-7)}}(\cdots (B^d_{p'_{1}}(\Upsilon_{1})\, ) \cdots )},  \\
\sum_{\ell_3 = 1}^{a_3 n} k_{\ell_3}^{(2)} &= \abs{ \mathcal{P}_{(1, 4, \ldots, d)}^{(\ell_1, \ell_{4}, \ldots, \ell_d)} \cap \mathcal{N}^1(B^{4}_{p'_{(2d-6)}}(\cdots (B^d_{p'_{2}}(\Upsilon_{2})\, ) \cdots )},
\end{align*}
and inequalities (\ref{projectionPigeonhole1}) and  (\ref{projectionPigeonhole2}).  In the last line, we merely took the intersection of the events, $E_{(\ell_4, \ldots, \ell_d)}^c$, over all $(\ell_4, \ldots, \ell_d)$, which are independent.
When combined, inequalities (\ref{projectionLB}) - (\ref{projectionUB2}), (\ref{clusterMergeBound}), and Lemma~\ref{planeBoundLemma} imply that for $d\geq 4$, the probability of $\vect{v_1}$ and $\vect{v_2}$ not being in the same component is at most $O(n^{-(d+1)})$.  This completes the proof of Lemma~\ref{clusterMergeLemma}.

\end{proof}


\subsection{The Size of the Giant Component}
\label{section_supercritical4}
To complete the proof of Theorem~\ref{thmb}, we need to demonstrate that the proportion of remaining vertices included in the giant component approaches $(1-q) >0$ in probability.  To this end we will prove Lemma~\ref{giantCompSizeLemma}, but first we will define $q$.

Recall from earlier that $q_i = \prob{\vect{Z}_t = \vect{0} \text{ for some } t \ \middle | \ \vect{Z}_0 = \vect{e_i} }{}$ is the extinction probability for a multitype branching process in which for any $k = 1, \ldots, d$, $(Z_1^{j} \ | \ \vect{Z}_0 = \vect{e_k}) \sim \text{Poisson}(\lambda a_j)$ if $j\neq k$ and $(Z_1^{k} \ | \ \vect{Z}_0 = \vect{e_k}) \equiv 0$, and initially there is one individual of type $i$.  The initial vertex in the cluster discovery process gives birth to $\text{Binomial}\left(a_i n, \frac{\lambda}{n}\right)$ neighbors of type $i$ for each $i$, and henceforth proceeds like the multitype process in which no vertex can give birth to its own type.  The limiting branching process is one in which each binomial birth event is replaced with a Poisson birth event with the same mean.  If we consider $(d-1)$ independent copies of this Poisson multitype branching process with the modified initial step, and we define $q$ to  be the extinction probability of one of these copies, then the collective process will have the same distribution for all time as the multitype branching process that starts with one individual of each of the $d$ types.  This implies that $\prod_i q_i = q^{(d-1)}$.  From the theory of multitype branching processes~\cite{branching}, the vector $(q_1, \ldots, q_d)$ is the solution to $\vect{f}(\vect{x}) = \vect{x}$ for $\vect{x} \in (0, 1)^d$, where
\begin{equation*}
f_i(\vect{x}) = \exp \left[ - \lambda \sum_{j \neq i} a_j (1 - x_j) \right].
\end{equation*}
Thus, we have implicitly defined $q <1$.

\begin{lem}
\label{giantCompSizeLemma}
\begin{equation}
\frac{\# \left\{\vect{v}\in V \ : \ \abs{C_{\vect{v}}} \leq \beta \log n, \xi_{\vect{v}} = 1 \right\} }{\lambda \left(\prod_i a_i \right) n^{(d-1)}} \ \longrightarrow \ q \hspace{1cm} \text{in probability.}
\end{equation}
\end{lem}

Our approach in the proof of Lemma~\ref{giantCompSizeLemma} will be to show that the probability of a vertex being included in a component of size at most $\beta \log n$ given that the vertex is occupied approaches $q$.  Then we will apply a second moment method argument to demonstrate that the actual proportion of occupied sites in small sized components approaches $q$.  This will require showing that the events $\{\abs{C_{\vect{v}}} \leq \beta \log n \}$ are asymptotically uncorrelated.

\begin{proof}[Proof of Lemma~\ref{giantCompSizeLemma}]
First we will show that $\prob{\abs{C_{\vect{v}}} \leq \beta \log n \ | \ \xi_{\vect{v}} = 1}{} \to q$.  Let $\vect{Z}_t$ be the multitype branching process with
\begin{equation*}
(Z_{t+1}^{j} \ | \ \vect{Z}_t = \vect{e_k}) \begin{cases}
\sim \text{Poisson}(\lambda a_j) & j\neq k \\
\equiv 0 & j = k
\end{cases}
\end{equation*}
for $k = 1, \ldots, d$ and $t\geq 1$, but with $Z_1^j \sim \text{Poisson}(\lambda a_j)$ for all $j = 1, \ldots, d$.  Then let $\displaystyle T := 1+ \sum_{t \geq 1} \norm{\vect{Z}_t}_1$ be the total size of this branching process, and let $\tau := \inf\{t \ | \ \norm{\vect{S}_t}_1 = 0\}$ be the total size of the upper bounding binomial branching process (recall that $\vect{S}_t$ is the random walk version of this branching process).  Also, let the Poisson branching process be coupled with the binomial branching process so as to minimize their total variation distance, $d_{TV}(Z_1^j, S_1^j) \leq \frac{a_j \lambda}{n}$~\cite{poisson}.  Thus,
\begin{align}
\label{tvd1}
\prob{T \leq \beta \log n, T \neq \tau, \tau \leq \beta \log n}{} \leq \frac{a_1 \lambda \beta \log n}{n}, \\
\label{tvd2}
\prob{T \leq \beta \log n, \tau > \beta \log n}{} \leq  \frac{a_1 \lambda \beta \log n}{n},
\end{align}
since in both events the two branching processes must differ in at least one of the first $\beta \log n$ birth events, and the probability of any birth event differing in the two branching processes is bounded by the total variation distance of their distributions (here we are referring to a randomly distributed number of children of a single type as a `birth event'). From the above discussion, we know that $\prob{T < \infty}{} = q$.  So we have:
\begin{align*}
\prob{\abs{C_{\vect{v}}} \leq \beta \log n \ | \ \xi_{\vect{v}} = 1}{}\ &\geq \ \prob{\tau \leq \beta \log n}{} \\
&\geq \ \prob{T \leq \beta \log n}{} - 2 \frac{a_1 \lambda \beta \log n}{n} \\
&\to \ q,
\end{align*}
which implies that $\liminf \prob{\abs{C_{\vect{v}}} \leq \beta \log n \ | \ \xi_{\vect{v}} = 1}{} \geq q$.  For the upper bound, we want to know that up to time $s = \beta \log n$, the processes $\vect{\A}_t$ and $\vect{S}_t$ are identical with high probability.  Let $T_{\vect{\A}} := \inf \{t \ | \ \vect{\A}_t = \vect{0} \}$.  Recall that the number of extra births in $\vect{S}_t$ that are lost in $\vect{\A}_t$ up to time $s \wedge T_{\vect{\A}}$ can be stochastically bounded above by $Y^{(s)} \sim \text{Binomial}(d\beta K (\log n)^2, \frac{\lambda}{n})$, as per (\ref{extraBirths}). Since, for $n > 2 \lambda d \beta K (\log n)^2$, 
\begin{align}
\prob{Y^{(s)} \geq 1}{} &= \sum_{k=1}^{d \beta K (\log n)^2} {d \beta K (\log n)^2 \choose k} \left(\frac{\lambda}{n} \right)^k \left(1 - \frac{\lambda}{n}\right)^{d \beta K (\log n)^2 -k} \nonumber \\
& \leq \sum_{k=1}^{d \beta K (\log n)^2} \left(\frac{\lambda d \beta K (\log n)^2}{n} \right)^k \nonumber \\
\label{noExtraBirths}
& \leq \frac{2 \lambda d \beta K (\log n)^2}{n},
\end{align}
the probability that the two processes will differ by time $s$ is small.  Now for the upper bound:
\begin{align*}
\prob{\abs{C_{\vect{v}}} \leq \beta \log n \ | \ \xi_{\vect{v}} = 1}{} &= \prob{ \vect{\A}_s = \vect{0}\ | \ \xi_{\vect{v}} = 1}{} \\
&\leq \prob{\vect{\A}_s = \vect{0}, \vect{\A}_t=\vect{S}_t \ \forall t \leq s \ | \ \xi_{\vect{v}} = 1}{}\\
& \hspace{1cm} + \prob{\vect{\A}_t \neq \vect{S}_t \text{ for some } t \leq s  \ | \ \xi_{\vect{v}} = 1}{} \\
& \leq \prob{\vect{S}_s = \vect{0}, \vect{\A}_t=\vect{S}_t \ \forall t \leq s  \ | \ \xi_{\vect{v}} = 1}{} + \prob{Y^{(s)} \geq 1}{} \\
&\leq \prob{\tau \leq \beta \log n}{} + \frac{2 \lambda d \beta K (\log n)^2}{n} \\
& \leq \prob{T \leq \beta \log n}{} + 2 \frac{a_1 \lambda \beta \log n}{n} + \frac{2 \lambda d \beta K (\log n)^2}{n} \\
& \to q.
\end{align*}
The fourth line above uses the estimate in (\ref{noExtraBirths}) and that $\vect{S}_s$ is independent of the vertex $\vect{v}$, and the fifth line uses the total variation distance bounds that we found earlier for the two branching processes.  Thus we have shown that $\limsup \prob{\abs{C_{\vect{v}}} \leq \beta \log n \ | \ \xi_{\vect{v}} = 1}{} = q$, and so
\begin{equation}
\label{probConv}
\lim_{n \to \infty} \prob{\abs{C_{\vect{v}}} \leq \beta \log n \ | \ \xi_{\vect{v}} = 1}{} = q.
\end{equation}

Now in order to complete the proof, we will apply the second moment method to the following random variables:
\begin{equation*}
H_{\vect{v}} = \begin{cases}
1 & \text{if }\abs{C_{\vect{v}}} \leq \beta \log n,\ \xi_{\vect{v}}=1 \\
0 & \text{else}.
\end{cases}
\end{equation*}
Note that 
\begin{align*}
\sum_{\vect{v} \in V} H_{\vect{v}} &= \# \left\{\vect{v}\in V \ : \ \abs{C_{\vect{v}}} \leq \beta \log n,\ \xi_{\vect{v}} = 1 \right\} \\
\sum_{\vect{v} \in V} \E H_{\vect{v}} &= \lambda \left(\prod_i a_i \right) n^{(d-1)}\ \prob{\abs{C_{\vect{v}}} \leq \beta \log n \ | \ \xi_{\vect{v}} = 1}{},
\end{align*}
so we wish to prove that
\begin{equation*}
\lim_{n\to \infty} \frac{\sum_{\vect{v} \in V} H_{\vect{v}}}{\sum_{\vect{v} \in V} \E H_{\vect{v}}} = 1
\end{equation*}
in probability.  To apply the second moment method we need to bound the variance of the sum of the $H_{\vect{v}}$'s, and to do this we need an upper bound on the covariance, $\E H_{\vect{v}} H_{\vect{w}} - (\E H_{\vect{v}} )^2$.  We begin by observing that
\begin{equation}
\label{covH}
\E H_{\vect{v}} H_{\vect{w}} = p^2 \ \prob{\abs{C_{\vect{v}}} \leq \beta \log n,\ \abs{C_{\vect{w}}} \leq \beta \log n \ | \  \xi_{\vect{v}}=1,\  \xi_{\vect{w}}=1}{}.
\end{equation}
If $\vect{w} \in \mathcal{N}(\vect{v})$ then
\begin{align*}
\prob{\abs{C_{\vect{v}}} \leq \beta \log n,\ \abs{C_{\vect{w}}} \leq \beta \log n \ | \  \xi_{\vect{v}}=1,\  \xi_{\vect{w}}=1}{} &= \prob{\abs{C_{\vect{v}}} \leq \beta \log n \ | \  \xi_{\vect{v}}=1,\  \xi_{\vect{w}}=1}{} \\
&\leq \prob{\abs{C_{\vect{v}}} \leq \beta \log n \ | \  \xi_{\vect{v}}=1}{}.
\end{align*}
Let us assume now that $\vect{w} \notin \mathcal{N}(\vect{v})$, and we will denote by $J_{\vect{v}} \subset V$ such that $\vect{v} \in J_{\vect{v}}$ an arbitrary set of vertices containing the vertex $\vect{v}$.  Then
\begin{align}
\prob{\abs{C_{\vect{v}}} \leq \beta \log n,\ \abs{C_{\vect{w}}} \leq \beta \log n \ | \  \xi_{\vect{v}}=1,\  \xi_{\vect{w}}=1}{}& \nonumber \\
& \nonumber \\
 = \sum_{\abs{J_{\vect{v}}} \leq \beta \log n} \ \sum_{\abs{J_{\vect{w}}} \leq \beta \log n} &\prob{C_{\vect{v}}= J_{\vect{v}},\ C_{\vect{w}} = J_{\vect{w}} \ | \ \xi_{\vect{v}}=1,\  \xi_{\vect{w}}=1}{} \nonumber \\
& \nonumber \\
\label{clustNoIntersect}
=  \sum_{\abs{J_{\vect{v}}} \leq \beta \log n} \ \sum_{\substack{ \abs{J_{\vect{w}}} \leq \beta \log n \\ J_{\vect{w}} \cap \mathcal{N}(J_{\vect{v}}) = \emptyset}} & \prob{C_{\vect{v}}= J_{\vect{v}},\ C_{\vect{w}} = J_{\vect{w}} \ | \ \xi_{\vect{v}}=1,\  \xi_{\vect{w}}=1}{} \\
\label{clustIntersect}
+ \sum_{\substack{\abs{J_{\vect{v}}} \leq \beta \log n \\ J_{\vect{v}} \ni \vect{w}}} &\prob{C_{\vect{v}}= J_{\vect{v}} = C_{\vect{w}} | \ \xi_{\vect{v}}=1,\  \xi_{\vect{w}}=1}{},
\end{align}
since if $J_{\vect{w}} \cap \mathcal{N}(J_{\vect{v}}) \neq \emptyset$ then the event that $C_{\vect{w}} = J_{\vect{w}}$ has zero probability unless $J_{\vect{w}} = J_{\vect{v}}$.  We will bound lines (\ref{clustNoIntersect}) and (\ref{clustIntersect}) separately.  Beginning with line (\ref{clustNoIntersect}), we condition on the event that $C_{\vect{v}} = J_{\vect{v}}$ to obtain:
\begin{align*}
\sum_{\abs{J_{\vect{v}}} \leq \beta \log n} \ \sum_{\substack{ \abs{J_{\vect{w}}} \leq \beta \log n \\ J_{\vect{w}} \cap \mathcal{N}(J_{\vect{v}}) = \emptyset}} \prob{C_{\vect{v}}= J_{\vect{v}} \ | \ \xi_{\vect{v}}=1,\  \xi_{\vect{w}}=1}{}  \prob{C_{\vect{w}}= J_{\vect{w}} \ | \ \xi_{\vect{v}}=1,\  \xi_{\vect{w}}=1, C_{\vect{v}} = J_{\vect{v}} }{}\\
= \sum_{\abs{J_{\vect{v}}} \leq \beta \log n} \prob{C_{\vect{v}}= J_{\vect{v}} \ | \ \xi_{\vect{v}}=1}{} \sum_{\substack{ \abs{J_{\vect{w}}} \leq \beta \log n \\ J_{\vect{w}} \cap \mathcal{N}(J_{\vect{v}}) = \emptyset}} \prob{C_{\vect{w}}= J_{\vect{w}} \ | \  \xi_{\vect{w}}=1, C_{\vect{v}} = J_{\vect{v}} }{}.
\end{align*}
The second line is because when $J_{\vect{w}} \cap \mathcal{N}(J_{\vect{v}}) = \emptyset$, the event that $C_{\vect{v}} = J_{\vect{v}}$ is independent of $\xi_{\vect{w}}$, and the event that $\xi_{\vect{v}} = 1$ is contained in the event that $C_{\vect{v}} = J_{\vect{v}}$.  Now, if $C_{\vect{v}} = J_{\vect{v}}$ and $J_{\vect{w}} \cap \mathcal{N}(J_{\vect{v}}) = \emptyset$ then we know that the vertices in $\mathcal{N}(J_{\vect{w}}) \cap \mathcal{N}(J_{\vect{v}})$ must not be occupied, otherwise $C_{\vect{v}}$ would have to include $J_{\vect{w}}$.  Additionally, these are the only vertices that both of the events $C_{\vect{v}} = J_{\vect{v}}$ and $C_{\vect{w}} = J_{\vect{w}}$ depend upon, otherwise these two events use independent vertices.  Noting that $\abs{\mathcal{N}(J_{\vect{w}}) \cap \mathcal{N}(J_{\vect{v}})} \leq 2 \abs{J_{\vect{v}}} \abs{J_{\vect{w}}}$, since each pair of non-neighboring vertices has at most two common neighbors, we have that (\ref{clustNoIntersect}) is bounded by:
\begin{align}
&\leq \sum_{\abs{J_{\vect{v}}} \leq \beta \log n} \prob{C_{\vect{v}}= J_{\vect{v}} \ | \ \xi_{\vect{v}}=1}{} \sum_{\substack{ \abs{J_{\vect{w}}} \leq \beta \log n \\ J_{\vect{w}} \cap \mathcal{N}(J_{\vect{v}}) = \emptyset}} \frac{ \prob{C_{\vect{w}}= J_{\vect{w}} \ | \  \xi_{\vect{w}}= 1}{} }{ \left(1 - \frac{\lambda}{n} \right)^{2 \beta^2 (\log n)^2 }} \nonumber \\
&\leq \left(1 + \frac{5 \lambda \beta^2 (\log n)^2}{n} \right) \sum_{\abs{J_{\vect{v}}} \leq \beta \log n} \prob{C_{\vect{v}}= J_{\vect{v}} \ | \ \xi_{\vect{v}}=1}{} \sum_{\abs{J_{\vect{w}}} \leq \beta \log n}  \prob{C_{\vect{w}}= J_{\vect{w}} \ | \  \xi_{\vect{w}}= 1}{} \nonumber \\
\label{noInstersectBd}
& =  \left(1 + \frac{5 \lambda \beta^2 (\log n)^2}{n} \right) \frac{(\E H_{\vect{v}})^2}{p^2}.
\end{align}
In the first line above we used that the event that none of the vertices in $\mathcal{N}(J_{\vect{w}}) \cap \mathcal{N}(J_{\vect{v}})$ are occupied is contained in the event that $C_{\vect{w}} = J_{\vect{w}}$.  In the second line we applied the following computation, in which the inequalities hold provided $n \geq 2 \lambda$ and $n \geq 32 \lambda \beta^2 (\log n)^2$:
\begin{align*}
\left( \frac{1}{1 - \frac{\lambda}{n}} \right)^{2 \beta^2 (\log n)^2} &\leq \left(1 + \frac{2\lambda}{n} \right)^{2 \beta^2 (\log n)^2} \\
&\leq 1 + \frac{4 \lambda \beta^2 (\log n)^2}{n} + \frac{32 \lambda^2 \beta^4 (\log n)^4}{n^2} \\
&\leq 1 + \frac{5 \lambda \beta^2 (\log n)^2}{n}.
\end{align*}

To bound line (\ref{clustIntersect}), we observe that
\begin{align}
\sum_{\substack{\abs{J_{\vect{v}}} \leq \beta \log n \\ J_{\vect{v}} \ni \vect{w}}} &\prob{C_{\vect{v}}= J_{\vect{v}} = C_{\vect{w}} | \ \xi_{\vect{v}}=1,\  \xi_{\vect{w}}=1}{} = \prob{\abs{C_{\vect{v}}} \leq \beta \log n,\ C_{\vect{v}} = C_{\vect{w}} \ | \ \xi_{\vect{v}} = 1,\ \xi_{\vect{w}} = 1 }{} \nonumber \\
&\leq \prob{A_{t} \cap \mathcal{N}(\vect{w}) \neq \emptyset \text{ for some } t\leq \beta \log n \ | \ A_{0} = \{ \vect{v} \} }{} \nonumber \\
&\leq 1 - \left(1 - \frac{\lambda}{n} \right)^{2\beta \log n} \nonumber \\
&\leq \frac{2 \lambda \beta \log n}{n} + \frac{8 \lambda^2 \beta^2 (\log n)^{2}}{n^2} \nonumber \\
\label{intersectBd}
& \leq \frac{3 \lambda \beta \log n}{n}
\end{align}
In words, the second line above says that the probability that $C_{\vect{v}}$ has fewer than $\beta \log n$ vertices and contains the vertex $\vect{w}$ (conditional on $\vect{v}$ and $\vect{w}$ begin occupied) is at most the probability that we find an occupied neighbor of $\vect{w}$ by time $\beta \log n$ during the process of discovering the cluster $C_{\vect{v}}$.  In the third line, we used the fact that at each time, $t\leq \beta \log n$, until $A_t \cap \mathcal{N}(\vect{w}) \neq \emptyset$ we have that $\abs{\mathcal{N}(\vect{v}_t ) \cap \mathcal{N}(\vect{w})} \leq 2$, and at least one such vertex must be occupied in order for $\vect{w}$ to be discovered.  The last two lines require that $n \geq 8 \lambda \beta \log n$.

Combining (\ref{covH}) -- (\ref{intersectBd}) shows that the covariance between $H_{\vect{v}}$ and $H_{\vect{w}}$ when $\vect{w} \notin \mathcal{N}(\vect{v})$ can be bounded as:
\begin{align}
\E H_{\vect{v}} H_{\vect{w}} - (\E H_{\vect{v}})^2 &\leq \frac{5 \lambda \beta^2 (\log n)^2}{n}(\E H_{\vect{v}})^2 + \frac{3 \lambda^3 \beta \log n}{n^3} \nonumber \\
\label{covBd}
& \leq \frac{6 \lambda^3 \beta^2 (\log n)^2}{n^3}
\end{align}
for sufficiently large $n$, since $\E H_{\vect{v}} \leq p$. Now we can use this to bound the variance:
\begin{align}
\Var \left(\sum_{\vect{v} \in V} H_{\vect{v}}\right) &= \sum_{\substack{\vect{v}, \vect{w} \in V \\ \vect{w} \in \mathcal{N}(\vect{v})}} \E H_{\vect{v}} H_{\vect{w}} + \sum_{\substack{\vect{v}, \vect{w} \in V \\ \vect{w} \notin \mathcal{N}(\vect{v})}} \E H_{\vect{v}} H_{\vect{w}} - \sum_{\vect{v}, \vect{w} \in V } \E H_{\vect{v}} \E H_{\vect{w}} \nonumber \\
&\leq \sum_{\substack{\vect{v}, \vect{w} \in V \\ \vect{w} \in \mathcal{N}(\vect{v})}} \E H_{\vect{v}} H_{\vect{w}} + \sum_{\substack{\vect{v}, \vect{w} \in V \\ \vect{w} \notin \mathcal{N}(\vect{v})}} \left[ \E H_{\vect{v}} H_{\vect{w}} -  \E H_{\vect{v}} \E H_{\vect{w}} \right] \nonumber \\
&\leq \lambda^2 \left(\prod_{i} a_i\right) a_1\, d\, n^{(d-1)} + 6\lambda^3 \beta^2 \left(\prod_{i} a_i\right)^2 n^{(2d-3)} (\log n)^2 \nonumber \\
\label{varBd}
&= O\left(n^{(2d-3)} (\log n)^2 \right).
\end{align}
Using this bound on the variance and Chebyshev's inequality yields:
\begin{align}
\prob{\abs{\frac{\sum_{\vect{v} \in V} H_{\vect{v}}}{\sum_{\vect{v} \in V} \E H_{\vect{v}}} - 1} > n^{-1/3}}{}&  \nonumber \\
& \hspace{-3cm} = \prob{\abs{\sum_{\vect{v} \in V} H_{\vect{v}} - \sum_{\vect{v} \in V} \E H_{\vect{v}} } > \lambda \left(\prod_i a_i \right) n^{(d-1)}\ \prob{\abs{C_{\vect{v}}} \leq \beta \log n \ | \ \xi_{\vect{v}} = 1}{} n^{-1/3} }{}  \nonumber \\
&\hspace{-3cm} \leq \frac{\Var \left(\sum_{\vect{v} \in V} H_{\vect{v}}\right)} {\lambda^2 \left(\prod_i a_i \right)^2 (q + o(1))^2 n^{(2d-2- 2/3)} } \nonumber \\
\label{proportionConv}
& \hspace{-3cm} = O\left(n^{-1/3} (\log n)^2 \right).
\end{align}
Combining (\ref{probConv}) and (\ref{proportionConv}) completes the proof of Lemma \ref{giantCompSizeLemma}.
\end{proof}
Thus the proof of Theorem~\ref{thmb} is complete.
\end{proof}


\section{Connectivity Threshold}
\label{section_connectivity}

We have demonstrated that when $p = \lambda/n$, a giant component emerges at $\lambda_c$.  The purpose of this section is to determine when the random site subgraph of the Hamming torus is connected.  We prove the following theorems.

\begin{theorem}[{\bf \ref{thm-discon} restated}]
Let $c < \frac{d-1}{\sum a_i}$.  If $p = p(n) \leq \frac{c \log n}{n}$ and $p = \omega(n^{-d})$ then the random site subgraph of the Hamming torus contains isolated vertices, and is thus not connected (a.a.s.).
\end{theorem}

\begin{theorem}[\bf \ref{thm-con} restated]
Let $c > \frac{d-1}{\sum a_i}$.  If $p = p(n) \geq \frac{c \log n}{n}$ then the random site subgraph of the Hamming torus is connected (a.a.s.).
\end{theorem}

\begin{theorem}[\bf \ref{thm-isocon} restated]
Fix $a_1 \geq a_2 \geq \cdots \geq a_d$, and let $c > \frac{d-1}{2\sum_{i=2}^d a_i + a_1 }$.  If $p = \frac{c \log n}{n}$ then every vertex in the random site subgraph of the Hamming torus is either isolated or belongs to the giant component (a.a.s.).
\end{theorem}

These three theorems together imply that, with probability approaching one, the random site subgraph of the Hamming torus is connected if and only if it contains no isolated vertices (except in the trivial case $p \asymp n^{-d}$, where the subgraph may consist of just a single vertex with positive probability).

\begin{proof}[Proof of Theorem~\ref{thm-discon}]  We denote by $d_{\vect{v}}$ the degree of the vertex $\vect{v}$ in the random site subgraph of the Hamming torus, with the convention that $d_{\vect{v}} = -1$ if $\xi_{\vect{v}} = 0$.  Let $I_n$ be the number of isolated vertices in the random site subgraph.  Then
\begin{align}
\nonumber \E I_n &= p \sum_{\vect{v} \in V} \prob{d_{\vect{v}} = 0 \ \middle | \ \xi_{\vect{v}} = 1}{} \\
\nonumber &= p \left(\prod a_i \right) n^d (1-p)^{\sum a_i n - d}\\
\label{numIsolated}
& \asymp p\, n^d \exp \left[ - p \sum a_i n \right],
\end{align}
where ``$a_n \asymp b_n$'' means that $0< \lim \frac{a_n}{b_n} < \infty$.  The last line above follows from the limit
\begin{equation*}
\lim_{n \to \infty} (1-p)^n e^{p n} = 1
\end{equation*}
provided $p = o(n^{-1/2})$.  If $n^{-1} \leq p \leq \frac{c \log n}{n}$, then equation (\ref{numIsolated}) is bounded by
\begin{equation*}
\E I_n \gae n^{d-1 - c \sum a_i} \to \infty.
\end{equation*}
Otherwise, if $n^{-d+1} \ll p \leq n^{-1}$, then equation (\ref{numIsolated}) yields
\begin{equation*}
\E I_n \gae p\, n^{d-1} e^{-\sum a_i} \to \infty.
\end{equation*}

We will now use the second moment method to show that $I_n$ stays close to its mean.
\begin{align*}
\E I_n^2 &= \sum_{\vect{v} \in V} \sum_{\vect{w} \in V} \prob{d_{\vect{v}} = 0,\ d_{\vect{w}} = 0}{} \\
&= \sum_{\vect{v} \in V} \prob{d_{\vect{v}} = 0}{} \    +      \sum_{\substack{\vect{v}, \vect{w} \in V \\ d(\vect{v}, \vect{w}) = 1}} \prob{d_{\vect{v}} = 0,\ d_{\vect{w}} = 0}{} \\
&\hspace{.8cm}      +      \sum_{\substack{\vect{v}, \vect{w} \in V \\ d(\vect{v}, \vect{w}) = 2}} \prob{d_{\vect{v}} = 0,\ d_{\vect{w}} = 0}{} \      +       \sum_{\substack{\vect{v}, \vect{w} \in V \\ d(\vect{v}, \vect{w}) \geq 3}} \prob{d_{\vect{v}} = 0,\ d_{\vect{w}} = 0}{}  \\ 
\\
&= \E I_n \     +   \   0  \      +        \sum_{\substack{\vect{v}, \vect{w} \in V \\ d(\vect{v}, \vect{w}) = 2}} \prob{d_{\vect{v}} = 0 \ \middle | \ d_{\vect{w}} = 0}{}\, \prob{d_{\vect{w}} = 0}{} \     +       \sum_{\substack{\vect{v}, \vect{w} \in V \\ d(\vect{v}, \vect{w}) \geq 3}} \prob{d_{\vect{v}} = 0}{} \prob{d_{\vect{w}} = 0}{}  \\
\\
&= \E I_n \     +   \      \sum_{\substack{\vect{v}, \vect{w} \in V \\ d(\vect{v}, \vect{w}) = 2}} \frac{\prob{d_{\vect{v}} = 0}{} \prob{d_{\vect{w}} = 0}{} }{ (1-p)^2} \     +       \sum_{\substack{\vect{v}, \vect{w} \in V \\ d(\vect{v}, \vect{w}) \geq 3}} \prob{d_{\vect{v}} = 0}{} \prob{d_{\vect{w}} = 0}{}  
\end{align*}
In the second line above, the second sum is equal to $0$ because if $\vect{v}$ and $\vect{w}$ are neighbors, then they cannot both be isolated.  The last sum in the second line is equal to the last sum in the third line because if $d(\vect{v}, \vect{w}) \geq 3$, then the events $\{ d_{\vect{v}}=0\}$ and  $\{d_{\vect{w}}=0 \}$ rely on disjoint sets of vertices and are thus independent.  In the case where $d(\vect{v}, \vect{w}) = 2$, these events share exactly two vertices in common, thus $\prob{d_{\vect{v}} = 0}{} =  \prob{d_{\vect{v}} = 0 \ \middle | \ d_{\vect{w}} = 0}{}\, (1-p)^2$.  Let $N_k = \# \{\vect{v} \in V \ | \ d((1, 1, \ldots, 1), \vect{v}) = k \}$ be the number of neighbors at exactly Hamming distance $k$ that a vertex in the Hamming torus has.  Then
\begin{align}
\nonumber \E I_n^2 &= \E I_n \    + \     \abs{V} N_2 \frac{\prob{d_{\vect{v}} = 0}{}^2}{(1-p)^2} \    + \     \abs{V} (\abs{V} - N_2 - N_1 - N_0)\, \prob{d_{\vect{v}} = 0}{}^2 \\
\nonumber &\lae  \E I_n \    + \     \abs{V} N_2 \prob{d_{\vect{v}} = 0}{}^2 \, (1 + 5p)\    + \     \abs{V} (\abs{V} - N_2)\, \prob{d_{\vect{v}} = 0}{}^2 \\
\nonumber &=  \E I_n \    + \    5p \abs{V} N_2 \prob{d_{\vect{v}} = 0}{}^2 \    + \     \abs{V}^2\, \prob{d_{\vect{v}} = 0}{}^2 \\
\nonumber &\leq (1 + O(\log^2 n))\, \E I_n \    + \    (\E I_n)^2,  \\
\nonumber \\
\label{varIsolated}
\Var(I_n) &=  O(\log^2 n)\, \E I_n.
\end{align}
In the second line above we applied the estimate $(1-p)^{-2} \leq (1+2p)^2 \leq (1 + 5p)$ whenever $p<1/4$.  In the fourth line we used the following facts: $\E I_n = \abs{V} \prob{d_{\vect{v}} = 0}{}$, $\prob{d_{\vect{v}} = 0}{} \leq p \leq \frac{c\log n}{n}$, and $N_2 = O(n^2)$.  We can now bound the fluctuations in $I_n$ by
\begin{align}
\prob{\abs{I_n - \E I_n} > (\E I_n)^{1/2} \log^2 n}{} \leq \frac{\Var(I_n)}{\E I_n \log^4 n} = O\left( ( \log n)^{-2} \right).
\end{align}
Since $\E I_n \gg n^{\epsilon} \gg \log^4 n$ for some $\epsilon>0$, this completes the proof of Theorem~\ref{thm-discon}.
\end{proof}

The proofs for Theorems~\ref{thm-con} and~\ref{thm-isocon} will be similar in that they appeal to the machinery that we have already developed in the proof of Theorem~\ref{thmb}.  In particular, we will merely replace Lemma~\ref{activeVerticesLemma} with a more assertive statement, then apply Lemmas~\ref{planeBoundLemma},~\ref{clusterMergeLemma} and the arguments of Section~\ref{section_supercritical1} with little modification.

\begin{proof}[Proof of Theorem~\ref{thm-con}]
Let $p = \frac{c \log n}{n}$, and observe that monotonicity implies the result for larger $p$.  Recall how we coupled a lower bounding random walk, $\vect{W}_t$, with $A'_t$ in Section~\ref{section_supercritical1}.  We wish to construct the same lower bounding walk here, and we will even use the same parameter $p_1 = \frac{\lambda_1}{n}$ where $\lambda_1$ is a constant such that $\lambda_c < \lambda_1 < \lambda_c^{(d-1)}$.  The only difference here is that now $p = p_1 + p_2 - p_1 p_2$ where $p_2 \geq \frac{\epsilon \log n}{n}$, which is okay because all we need to apply Lemma~\ref{clusterMergeLemma} is that $p_2 \geq \frac{\epsilon}{n}$.  We then can apply  Lemma~\ref{planeBoundLemma} and repeat the arguments at the end of Section~\ref{section_supercritical1} to again obtain equation~(\ref{survival-r}), which says that, provided $\abs{A'_0} \geq m = \frac{d+1}{\gamma_2} \log n$, $\abs{A'_0} = o(n^{2/3})$ and $\abs{R'_0} = o(n^{2/3})$, the process of discovering occupied vertices with parameter $p_1$ will survive until time $r = n^{d-4/3}$.  At this point, we can apply Lemma~\ref{clusterMergeLemma} to say that with high probability every cluster reaching size $r$ will connect into a single giant component.  Thus, to show that the entire subgraph is connected, we need to show that either $\abs{C_{\vect{v}}} \geq m$ or $\abs{C_{\vect{v}}} = 0$ for every $\vect{v} \in V$ with probability approaching $1$.

Consider a fixed vertex, $\vect{v}$.  Disregarding prior notation for the moment, let $A_k$ be the set of vertices in the random site subgraph of the Hamming torus to which there exists a unique shortest path of length $k$ from $\vect{v}$.  
%
%
Then for any constant $\ell$ 
\begin{align}
\nonumber \prob{\abs{A_2} < m \ \middle | \ \abs{A_1} = \ell}{p} &= \prob{e^{-\abs{A_2}} > e^{-m} \ \middle | \ \abs{A_1} = \ell}{p} \\
\nonumber & \leq e^m \E_p \left[e^{-\abs{A_2}}\ \middle | \ \abs{A_1} = \ell \right] \\
\nonumber & = e^m \left[ 1 - p(1 - e^{-1}) \right]^{\abs{\mathcal{N}(A_1) \setminus \mathcal{N}(\vect{v}) }} \\
\nonumber &\leq \exp \left[ m - p \left(1-e^{-1}\right) \abs{\mathcal{N}(A_1) \setminus \mathcal{N}(\vect{v}) } \right] \\
\nonumber &\leq \exp \left[ m - \frac{c \log n}{n} \left(1-e^{-1}\right) \left(\sum_{i=2}^{d} a_i n - \ell - d + 1\right) \ell \right] \\
\label{twoStepBd1}
 & = \exp \left[ \left(\frac{d+1}{\gamma_2}  - c \left(1-e^{-1}\right) \sum_{i=2}^{d} a_i \right) \log n +  c(1-e^{-1}) \ell (\ell + d -1) \frac{\log n}{n} \right].
\end{align}
In the fourth line above we used the bound $(1-x)\leq e^{-x}$.  In the fifth line above we assumed WLOG that $a_1 \geq a_2 \geq \cdots \geq a_d$, and used the fact that each of the $\ell$ vertices in $A_1$ has at least $(a_2 + \cdots + a_d)n - \ell - d +1$ neighbors that are not also neighbors of $\vect{v}$ or any other vertex in $A_1$.  Now, if we choose $\ell$ to be any constant so that
\begin{align}
\nonumber \frac{d+1}{\gamma_2}  - \ell c \left(1-e^{-1}\right) \sum_{i=2}^{d} a_i  \leq -(d+1) \\
\label{ellDef}
\ell \geq \frac{(d+1)(1 + 1/\gamma_2)}{c \left(1-e^{-1}\right)  \sum_{i=2}^{d} a_i },
\end{align}
then (\ref{twoStepBd1}) implies that
\begin{equation}
\label{twoStepBd2}
\prob{\abs{A_2} < m \ \middle | \ \abs{A_1} = \ell}{p} = O\left(n^{-(d+1)}\right).
\end{equation}
Let $\vect{Z}_t$ be the multitype branching process with $Z_1^i \sim \text{Binomial}(a_i n, p)$ for all $i = 1, \ldots, d$, and for all $t\geq 1$, $(Z_{t+1}^j \ | \ \vect{Z}_{t} = \vect{e_i}) \sim \text{Binomial}(a_j n, p)$ for all $i$ and $j\neq i$, and $(Z_{t+1}^i \ | \ \vect{Z}_{t} = \vect{e_i}) \equiv 0$.  Observe that $\abs{A_k} \leq \norm{\vect{Z_k}}_1$ by a now familiar coupling.  For $\theta = \log 5 \approx 1.6$
\begin{align}
\nonumber \prob{\norm{\vect{Z_1}}_1 \geq 5 \sum a_i c \log n}{} &\leq e^{-5\theta \sum a_i c \log n} \E e^{\theta \norm{\vect{Z_1}}_1} \\
\nonumber & \leq e^{-5\theta \sum a_i c \log n} \left[ 1 + p \left( e^\theta -1 \right) \right]^{\sum a_i n} \\
\nonumber & \leq \exp \left[ -5\theta \sum a_i c \log n \, + \,  \left( e^\theta -1 \right) \sum a_i c \log n\right] \\
\nonumber & \leq \exp \left[ \left(e^{\theta} - 1 - 5\theta \right) \sum a_i c \log n \right]\\
\nonumber & \leq \exp \left[ -4 (d-1) \log n \right]\\
\label{twoStepUpperBd1}
&= O(n^{-(d+2)}).
\end{align}
Now, applying a union bound,
\begin{align}
\nonumber \prob{\norm{\vect{Z_2}}_1 \geq \left (5 \sum a_i c \log n \right)^2 \ \middle | \ \norm{\vect{Z_1}}_1 \leq 5 \sum a_i c \log n}{} &  \\
\nonumber & \hspace{-4cm} \leq \left(5 \sum a_i c \log n\right) \prob{\norm{\vect{Z_1}}_1 \geq 5 \sum a_i c \log n}{} \\
\label{twoStepUpperBd2}
& \hspace{-4cm} = O(n^{-(d+1)}).
\end{align}
Combining the last two inequalities shows that $\abs{A_1}$ and $\abs{A_2}$ are both $o(n^{2/3})$ with high probability.  So we can apply Lemma~\ref{planeBoundLemma} with $A'_0 = A_2$, $R'_0 = A_1 \cup \{\vect{v}\}$ and $U'_0 = V \setminus \mathcal{N}(R'_0)$, and continue the arguments discussed above to show that all vertices for which $d_{\vect{v}} \geq \ell$ will be in the same large component with probability exceeding $1-O(n^{-1})$ (by a union bound over all vertices in $V$).  Thus, we need to show that all vertices either have $d_{\vect{v}} = -1$ or $d_{\vect{v}}\geq \ell$.
\begin{align*}
\prob{d_{\vect{v}} < \ell \ \middle | \ d_{\vect{v}} \neq -1}{} &= \sum_{j = 0}^{\ell-1} { \sum a_i n - d \choose j} \left( \frac{c \log n}{n} \right)^j  \left( 1- \frac{c \log n}{n} \right)^{\sum a_i n - d-j} \\
& \leq \sum_{j = 0}^{\ell-1} \frac{(c \log n)^j}{j!} \left(\sum a_i \right)^j e^{- \sum a_i c \log n + (d + j)c \log n / n} \\
& \leq \ell \, (c \log n)^{\ell-1} \, n^{-\sum a_i c} \, e^{(d + \ell - 1)c \log n / n} \, \max \left\{ 1, \left(\sum a_i\right)^{\ell -1}\right\} \\
& = O(n^{-\sum a_i c} (\log n)^{\ell-1}).
\end{align*}
So, since $\sum a_i c > d-1$, we have that
\begin{equation*}
\prob{0 \leq d_{\vect{v}} < \ell}{} = O(n^{-\sum a_i c - 1} (\log n)^{\ell}) = o(n^{-d}),
\end{equation*}
and by the union bound
\begin{equation*}
\prob{0 \leq d_{\vect{v}} < \ell \text{ for some } \vect{v} \in V}{} = o(1).
\end{equation*}
This completes the proof of Theorem~\ref{thm-con}.
\end{proof}

\begin{proof}[Proof of Theorem~\ref{thm-isocon}]
The proof here will be very similar to that of Theorem~\ref{thm-con}, except that here we will show that with probability approaching 1 one of three things can occur for each $\vect{v} \in V$: $\abs{C_{\vect{v}}} = 0$, $\abs{C_{\vect{v}}} = 1$ or $\abs{C_{\vect{v}}} \geq m$.  The last case implies, by the discussion at the start of the last proof, that $\vect{v}$ is in the giant component.  We will also rely on the same notation from the last proof.

If $d_{\vect{v}} \leq 0$ then $\vect{v}$ is either isolated or not in the subgraph, so we consider the cases where $d_{\vect{v}} \geq 1$.  If $d_{\vect{v}} \geq \ell$, where $\ell$ is defined according to~(\ref{ellDef}), then equations~(\ref{twoStepBd1}) and~(\ref{twoStepBd2}) still hold, as well as the arguments that follow, so $\vect{v}$ will be in the giant component with probability approaching one.  Now we must consider the case where $1 \leq d_{\vect{v}} < \ell$.  Suppose $d_{\vect{v}} = l$ where $1\leq l < \ell$, then
\begin{align}
\nonumber \prob{\abs{A_2} < \ell,\ d_{\vect{v}} =l}{} & \\
\nonumber & \hspace{-2.5cm} = p \, {\sum a_i n \choose l} \, p^l (1-p)^{\sum a_i n - d - l} \ \sum_{j=0}^{\ell - 1} {\abs{\mathcal{N}(A_1) \setminus \mathcal{N}(\vect{v})} \choose j} p^j (1-p)^{\abs{\mathcal{N}(A_1) \setminus \mathcal{N}(\vect{v})}-j } \\
\nonumber & \hspace{-2.5cm} \leq \sum_{j=0}^{\ell -1} \left(l \sum a_i \right)^{j+l} \frac{(c \log n)^{j+l+1} }{j!}\, n^{-c((l+1)\sum_{i=2}^d a_i + a_1) - 1} \, e^{(j + l (l+d -1))c \log n / n} \\
\nonumber & \hspace{-2.5cm} \leq \ell \max\left\{  \left( \sum a_i \right)^{l},\, \left( \sum a_i \right)^{\ell+l} \right\} (c \log n)^{\ell+1}\, n^{-c((l+1)\sum_{i=2}^d a_i + a_1) - 1} \, e^{(\ell + l(l+d -1))c \log n / n} \\
\label{threeStepBd1}
& \hspace{-2.5cm} = o(n^{-d}).
\end{align}
The second line above is because each of the $l$ vertices in $A_1$ have at least $(a_2 + \cdots + a_d) n - l - d + 1$ neighbors that are not also neighbors of $\vect{v}$ or any other vertex in $A_1$.  The last line above is because $c > \frac{d-1}{2\sum_{i=2}^d a_i + a_1 }$.  Summing the above probabilities implies that $\prob{\abs{A_2} < \ell,\ 1 \leq d_{\vect{v}} < \ell}{} = o(n^{-d})$.  So for every $\vect{v} \in V$, if $d_{\vect{v}} \leq 0$, then $\vect{v}$ is either isolated or not in the random subgraph.  If $d_{\vect{v}} \geq \ell$, then $\vect{v}$ is in the giant component with probability approaching one.  The probability of none of the last three cases occurring and $\vect{v}$ having fewer than $\ell$ occupied vertices at distance $2$ approaches $0$.  To complete the proof, we merely need to show that if $\vect{v}$ has $\ell$ occupied vertices at distance $2$ then $\vect{v}$ will have at least $m$ vertices at distance $3$.  This is easy because the arguments are identical to those of inequalities~(\ref{twoStepBd1}),~(\ref{twoStepBd2})-(\ref{twoStepUpperBd2}), but with $\abs{A_2}$ and $\abs{A_1}$ replaced by $\abs{A_3}$ and $\abs{A_2}$, and $\ell$ replaced by $4\ell$ in inequality~(\ref{twoStepBd1}).  This last substitution is because if $\abs{A_1}< \ell$ and $\abs{A_2} = \ell$ then each vertex in $A_2$ has at least $(a_2 + \cdots + a_d) n - 4\ell - d + 1$ neighbors that are not also neighbors of $\vect{v}$ or any other vertex in $A_1$ or $A_2$.  This completes the proof of Theorem~\ref{thm-isocon}.

\end{proof}

	\bibliographystyle{abbrv}
	\bibliography{proposal_cite}
  
\end{document}

%% file: clusterMerge.pdf_t
\begin{picture}(0,0)%
\includegraphics{clusterMerge.pdf}%
\end{picture}%
\setlength{\unitlength}{3947sp}%
\begingroup\makeatletter\ifx\SetFigFontNFSS\undefined%
\gdef\SetFigFontNFSS#1#2#3#4#5{%
  \reset@font\fontsize{#1}{#2pt}%
  \fontfamily{#3}\fontseries{#4}\fontshape{#5}%
  \selectfont}%
\fi\endgroup%
\begin{picture}(4204,2736)(2686,-4714)
\put(3301,-3211){\makebox(0,0)[lb]{\smash{{\SetFigFontNFSS{12}{14.4}{\familydefault}{\mddefault}{\updefault}{$v_2$}%
}}}}
\put(4876,-4636){\makebox(0,0)[lb]{\smash{{\SetFigFontNFSS{12}{14.4}{\familydefault}{\mddefault}{\updefault}{$v_1$}%
}}}}
\put(2701,-3886){\makebox(0,0)[lb]{\smash{{\SetFigFontNFSS{12}{14.4}{\familydefault}{\mddefault}{\updefault}{$v_3$}%
}}}}
\put(4951,-2161){\makebox(0,0)[lb]{\smash{{\SetFigFontNFSS{12}{14.4}{\familydefault}{\mddefault}{\updefault}{$v_3 = \ell_3$}%
}}}}
\end{picture}%